\documentclass[11pt]{amsart}
\usepackage{amssymb}	
\usepackage{tikz}
\usepackage{tikz-cd}

\usepackage{amsmath}
\usepackage{mathrsfs}
\usepackage{amsthm}
\usepackage{verbatim}
\usepackage{graphicx}
\usepackage{color}
\usepackage{float}
\usepackage{algorithm}
\usepackage{algpseudocode}
\usepackage{hyperref}
\usepackage[
labelfont=sf,
hypcap=false,
format=hang,
]{caption}

\usepackage{bm}
\usepackage{amsfonts}
\usepackage{amscd}
\usepackage{cases}
\usepackage{xcolor}
\usepackage{subeqnarray}

\newtheorem{theorem}{Theorem}
\newtheorem{remark}{Remark}

\newtheorem{lemma}{Lemma}

\newtheorem{proposition}{Proposition}

\newcommand\grad{\operatorname{grad}}
\renewcommand\div{\operatorname{div}}

\newcommand\curl{\operatorname{curl}}

\newcommand\bu{\bm{u}}
\newcommand\bv{\bm{v}}

\newcommand \figurepath{./figures/}
\newcommand \moviepath{./movie/}

\newtheorem{innercustomthm}{Algorithm}
\newenvironment{customthm}[1]
  {\renewcommand\theinnercustomthm{#1}\innercustomthm}
  {\endinnercustomthm}

\usepackage{chngpage}

\usepackage{lipsum}

\usepackage{booktabs}

\makeatletter
\@addtoreset{equation}{section}
\makeatother

\author[K. Hu]{Kaibo Hu}
\address{School of Mathematics, University of Minnesota,
 206 Church St. SE,
Minneapolis, MN 55455-0488,
USA}
\email{kaibo.hu@maths.ox.ac.uk}

\author[Y.-J. Lee]{Young-Ju Lee}
\address{Department of Mathematics, Texas State University, TX, USA}
\email{yjlee@txstate.edu}

\author[J. Xu]{Jinchao Xu}
\address{Department of Mathematics, Pennsylvania State University, University Park, PA 16802, USA}
\email{xu@math.psu.edu}

\begin{document}
\title[Helicity-conservative FEM for MHD]{Helicity-conservative finite element discretization for incompressible MHD systems}

\maketitle

\begin{abstract}
We construct finite element methods for the incompressible magnetohydrodynamics (MHD) system that precisely preserve magnetic and cross helicity, the energy law and the magnetic Gauss law at the discrete level.  The variables are discretized as discrete differential forms  in a de~Rham complex. We present numerical tests to show the performance of the algorithm. 
\end{abstract}

\smallskip
\noindent \textbf{Keywords:} magnetohydrodynamics, helicity, divergence-free, structure-preserving, finite element.




\section{Introduction} 

Numerical simulation for the incompressible magnetohydrodynamics (MHD) system is important in plasma physics. There have been a lot of efforts in designing stable and efficient numerical methods for solving the MHD equations. 

The MHD system has various conserved quantities.
Among them, the energy law and the magnetic Gauss law ($\nabla\cdot \bm{B}=0$) have been proved crucial both for the MHD physics and for computation, c.f., \cite{brackbill1980effect}. Moreover, topology of magnetic and fluid fields plays an important role in many applications of MHD. The linking and knot structures of the magnetic field are rearranged in magnetic reconnection and this fact leads to a number of consequences in physics.  The helicity of a divergence-free vector field, which is conserved in non-dissipative systems (ideal flows), is a standard measure for the extent to which the field lines wrap and coil around one another \cite{cantarella1999influence}. Fluid and MHD helicity is known to be important in the turbulence regime as discussed in, for example, \cite{frisch1975possibility,perez2009role}. 
Helicity also provides a local lower bound for the energy \cite[p. 122]{arnold1999topological}, i.e., a topological obstruction of energy relaxation.
 We refer to \cite{arnold1999topological,berger1984topological,moffatt1981some,moffatt1992helicity,moffatt2014helicity} and the references therein for more discussions on MHD helicity, and \cite{smiet2017knots} for discussions on knots in plasma physics.  

In many algorithms, these conservation laws are only approximated up to a discretization error, rather than exactly conserved. These approximation errors may then pollute the solution with unphysical behavior.
It is therefore of great interest to construct numerical methods that precisely preserve the helicity up to the machine precision, together with other conservative quantities including the energy and the magnetic Gauss law. These conservation laws are related to each other. For example, to obtain well defined magnetic helicity $\int \bm{A}_{h}\cdot \bm{B}_{h}\, \mathrm{d}x$ at the discrete level, where $\bm{A}_{h}$ is any magnetic potential of $\bm{B}_{h}$ satisfying $\nabla\times \bm{A}_{h}=\bm{B}_{h}$, the discrete magnetic field $\bm{B}_{h}$ has to be precisely divergence-free.  As a consequence, it is necessary to use algorithms preserving the magnetic Gauss law. 

While it is very difficult to provide a complete review of the huge literature on MHD simulations, let us only mention some early work on finite element methods \cite{gunzburger1991existence,schotzau2004mixed}, and recent work on finite element methods that preserve the energy law and the magnetic Gauss law at the discrete level \cite{hiptmair2018fully,hiptmair2018splitting,hu2014stable}.
    However, neither magnetic nor cross helicity is preserved in most, if not all, of these finite element schemes even in a semi-discretization with continuous time. We refer to Section \ref{helicity-polution} below and \cite{phdthesis} for an analysis of the artificial helicity pollution in one of these schemes (note that many schemes do not preserve the precise divergence-free condition of the magnetic field, so the helicity is even not defined for the numerical solutions).
  

On the other hand, we mention some existing efforts on helicity-preserving schemes. Liu and Wang \cite{liu2004energy} studied helicity-preserving finite difference methods for axisymmetric Navier-Stokes and MHD flows.  Rebholz \cite{layton2008helicity,rebholz2007energy} constructed energy- and helicity-preserving finite element methods for the Navier-Stokes equations.  See \cite{olshanskii2010note} for further discussions on, e.g.,  turbulence models. Kraus and Maj \cite{kraus2017variational} studied helicity-preserving schemes for the MHD system based on discrete exterior calculus.   To the best of our knowledge, however, it remains open to construct helicity-preserving finite element methods for MHD.

The goal of this paper is to construct finite element methods for the incompressible MHD system preserving the energy, the magnetic Gauss law and the magnetic and cross helicity at the discrete level in the ideal MHD limit.  
 Actually, the proposed scheme preserves the local helicity as well, meaning that the helicity is conserved on any subdomain of $\Omega$ if the field has certain vanishing conditions on its boundary. This in turn provides a lower bound for the local discrete energy, i.e., a topological obstruction of energy relaxation at the discrete level. As demonstrated by the numerical results, the proposed algorithm also shows a significant improvement on the approximation of helicity even for the resistive incompressible MHD models, compared to popular schemes that are not constructed with an emphasis on structure-preservation. We remark that a slightly different version of the local helicity is defined by the Woltjer invariant, 
 $$
 \mathcal{W}_{U}:=\int_{\phi_{t}(U)}\bm{A}\cdot \bm{B}\, dx,
 $$
 where $\phi_{t}$ is the flow of the velocity vector field, $U$ is any subdomain of $\Omega$ and $\bm{B}\cdot \bm{n}_{\partial U}=0$ on $\partial U$. In this paper, we only consider fixed domains and will not deal with the Woltjer invariant where the domain is dragged by the flow.

To preserve several conserved quantities at the discrete level, we adopt a discrete differential form point of view. 
Comparing to recent efforts \cite{gawlik2019variational,gawlik2020conservative} on structure-preserving discretization for the fluid mechanics  with $H(\div)$-conforming velocity, we use an $H(\curl)$-based formulation for the fluid part to preserve the helicity. A similar discretization for the Navier-Stokes equations based on the N\'ed\'elec edge element can be found in \cite{girault1990curl}.  However, helicity-preservation was not addressed there.
Comparing with existing works for the MHD discretization, e.g., \cite{hiptmair2018fully,hu2014stable}, we introduce the discrete Hodge dual ($L^{2}$ projections) of the magnetic field and the vorticity as independent variables in proper spaces. 
As a summary, we use $(\bu_{h}, \bm{\omega}_{h}, \bm{j}_{h}, \bm{E}_{h}, \bm{H}_{h}, \bm{B}_{h}, P_{h}) \in  [H^{h}_{0}({{\curl}},\Omega)]^5 \times H^{h}_{0}(\div,\Omega) \times H_0^h(\rm grad)$ as variables, where $\bu_{h}$ is the velocity, $\bm{\omega}_{h}$ is the vorticity, $\bm{j}_{h}$ is the current density, $\bm{E}_{h}$ is the electric field, $\bm{H}_{h}$ is the magnetizing field, $\bm{B}_{h}$ is the magnetic field, and $P_{h}$ is the total pressure, respectively. Here we introduce the vorticity variable $\bm{\omega}_{h}\in H_{0}^{h}(\curl, \Omega)$ independent of $\nabla\times \bm{u}_{h}\in H_{0}^{h}(\div, \Omega)$ and the magnetic variable $\bm{H}_{h}\in H_{0}^{h}(\curl, \Omega)$ independent of $\bm{B}_{h}\in H_{0}^{h}(\div, \Omega)$. Together with carefully designed discrete variational forms (\eqref{main:eq3d} below), these choices of unknowns and spaces are the key of the helicity conservation.  The algorithm is valid for unstructured meshes on general domains and can be of arbitrary order in the framework of the finite element exterior calculus (FEEC) \cite{arnold2018finite,Arnold.D;Falk.R;Winther.R.2006a}.

 The resulting system has more variables than the scheme in, e.g., \cite{hu2014stable}, but is still easy to solve. In the numerical tests, we use iterative methods to solve the algebraic systems.

The rest of the paper is organized as follows. In Section \ref{sec:preliminary} we provide preliminaries and notation. In Section \ref{sec:alg} we present our algorithm that preserves the helicity. 
In Section \ref{sec:num}, we present numerical results on the convergence and helicity-preserving properties of our algorithms. In Section \ref{sec:con}, we give some concluding remarks. In the Appendix, we show the well-posedness of the scheme presented in Section \ref{sec:alg}.

\section{Preliminaries} \label{sec:preliminary}

\subsection{Notation}

Helicity can be defined on 3D contractible domains.  One can extend the definition to nontrivial topology and different space dimensions  \cite[Chapter 3]{arnold1999topological}. However, for simplicity of presentation, in this paper we assume that $\Omega\subset \mathbb{R}^{3}$ is a contractible bounded Lipschitz domain. 

We use the standard notation for the inner product and the norm of the
{$L^{2}$} space
$$
(u,v):=\int_{\Omega}u\cdot v\, {d}x,\quad
\|u\|:=\left(\int_{\Omega} \lvert u\rvert^2\, {d}x\right)^{1/2}.
$$
 Define the following $H(D,\Omega)$ space with a given linear operator {$D$}, which is either $\grad$, $\curl$, or $\div$:
$$
H(D,\Omega):=\{v\in L^2(\Omega), Dv\in L^2(\Omega)\},  
$$
and
$$
H_0(D,\Omega):=\{v\in H(D, \Omega), t_{D}v=0 \mbox{ on } \partial\Omega\},
$$
where $t_{D}$ is the trace operator:
$$
t_{D}v:=
\left\{
  \begin{array}{cc}
    v, & D=\mathrm{grad},\\
    v\times n, & D=\mathrm{curl},\\
    v\cdot n, & D=\mathrm{div}.
  \end{array}
\right.
$$
We also define:
$$
L^2_0(\Omega):=\left\{v\in L^2(\Omega):  \int_\Omega v\, dx=0 \right\}.
$$
By definition, $H_0(\mathrm{grad}, \Omega)$ coincides with $H^1_0(\Omega)$.

The de~Rham complex in three space dimensions with vanishing boundary conditions reads:
 \begin{equation}\label{sequence3-0}
\begin{tikzcd}
0 \arrow{r}  \arrow{r} &H_{0}({\grad}, \Omega)  \arrow{r}{\grad}& {H}_{0}(\curl, \Omega)\arrow{r}{\curl} &{H}_{0}(\mathrm{div}, \Omega)\arrow{r}{\div}& L_{0}^{2}(\Omega) \arrow{r}{} & 0
 \end{tikzcd}
\end{equation}
The sequence \eqref{sequence3-0} is exact on contractible domains, meaning that  $\mathcal{N}(\curl)=\mathcal{R}(\grad)$ and $\mathcal{N}(\div)=\mathcal{R}(\curl)$, where $\mathcal{N}$ and $\mathcal{R}$ denote the kernel and range of an operator, respectively. 

The main idea of the discrete differential forms \cite{arnold2018finite,Arnold.D;Falk.R;Winther.R.2006a} or the finite element exterior calculus \cite{Bossavit.A.1998a,Hiptmair.R.2002a} is to construct finite element discretizations for the spaces in \eqref{sequence3-0} such that they fit into a discrete sequence: 
\begin{equation*}\label{sequence3-0-h}
{\begin{tikzcd}
0 \arrow{r}  &H_{0}^{h}({\grad}, \Omega)  \arrow{r}{\grad}& {H}_{0}^{h}(\curl, \Omega)\arrow{r}{\curl} &{H}^{h}_{0}(\mathrm{div}, \Omega)\arrow{r}{\div}& L_{0}^{2, h}(\Omega) \arrow{r}{} & 0
 \end{tikzcd}} 
\end{equation*}
The discrete de~Rham sequences can be  of arbitrary order \cite{Arnold.D;Falk.R;Winther.R.2006a,Boffi.D;Brezzi.F;Fortin.M.2013a}. 
We use $\mathbb{Q}_h^{\rm curl}$, the $L^2$ projection to $H_0^h({\curl},\Omega )$: 
\begin{equation}
\mathbb{Q}^{\rm curl}_h: \left [ L^{2}(\Omega)\right ]^{3}\to H^{h}_{0}({{\curl}},\Omega),
\end{equation} 
and the discrete curl operator $\nabla_h \times : H^h_0({\rm div}, \Omega) \to H_0^h({{\curl}}, \Omega)$ defined by the following relation: 
\begin{equation}
(\nabla_h \times \bm{U}, \bm{V}) = (\bm{U}, \nabla \times \bm{V}), \quad \forall (\bm{U},\bm{V}) \in H_0^h({\rm div},\Omega) \times H_0^h({{\curl}},\Omega).  
\end{equation} 
We shall also use the $L^2$ adjoint operator of the discrete gradient, i.e., the discrete divergence operator $\nabla_{h} \cdot : H^{h}_{0}(\curl, \Omega)\to H^{h}_{0}(\grad, \Omega)$ defined by 
\begin{equation}
(\nabla_{h}\cdot \bm{v}_{h}, \phi_{h}):=-(\bm{v}_{h}, \nabla \phi_{h}), \quad \forall (\bm{v}_{h}, \phi_{h}) \in H^{h}_{0}(\curl, \Omega)\times H^{h}_{0}(\grad, \Omega). 
\end{equation} 

Let $V_{h}$ be any subspace of $L^{2}$ and $\mathbb{P}_{V_{h}}: L^{2}\to V_{h}$ be the $L^{2}$ projection. Then we will frequently use the following property: 
\begin{equation}\label{move-P}
(\mathbb{P}_{V_{h}}u, v)=(u, \mathbb{P}_{V_{h}}v), \quad \forall \, u, v\in L^{2}.
\end{equation}
For any time-dependent function, we use a subscript  $t$ to denote its time derivative. For example, $(\bm{u}_{h})_{t}$ is the time derivative of the velocity if $\bm{u}_{h}$ is the velocity field.

\subsection{MHD equations and conserved quantities}\label{sec:intro}
Consider the following system of equations in $\Omega \times (0, {T}]$: 
\begin{subeqnarray}\label{main:eq2} 
\partial_t \bu - \bu\times \bm{\omega} + R_{e}^{-1}\nabla \times \nabla \times \bu - \textsf{c} \bm{j} \times \bm{B} + \nabla P &=& \bm{f},  \slabel{main:eq2m} \\ 
\bm{j} - \nabla \times \bm{B}  &=& \bm{0}, \\
\partial_t \bm{B} + \nabla \times \bm{E} &=& \bm{0}, \slabel{main:eqf} \\ 
R_m^{-1} \bm{j} -  \left ( \bm{E} + {\bf{u}} \times \bm{B} \right ) &=& \bm{0}, \slabel{main:eqo} \\ 
\nabla \cdot {\bf{u}} &=& 0,
\end{subeqnarray}
where $\partial_t {\bf{u}} = \partial {\bf{u}} / \partial t$  and $\partial_t \bm{B}=\partial \bm{B}/\partial t$ are the time derivatives of $\bu$ and $\bm{B}$; ${\bf{u}}$ is the fluid velocity; and $\bm{j}$, $\bm{B}$ and $\bm{E}$ are the volume current density, the magnetic field and the electric field, respectively. In \eqref{main:eq2m}, we use the total pressure $P:=1/2|\bm{u}|^{2}+p$ as an unknown, where $p$ is the physical pressure.  
The fluid momentum equation \eqref{main:eq2m} is in the Lamb form \cite{lamb1932hydrodynamics} with the vorticity $\bm{\omega}:=\nabla\times \bm{u}$. In \eqref{main:eq2}, $\bm{j} \times \bm{B}$ is called the Lorentz force, the force that the magnetic field exerts on the conducting fluid, and $\textsf{c}:=V_{A}^{2}/V^{2}$ is the coupling number, where $V_{A}$ and $V$ are the scales of the {{Alfv\'en}} speed and of the flow, respectively.  The parameters $R_{e}$ and $R_{m}$ are the fluid and magnetic Reynolds numbers, respectively. Throughout this paper, we will refer to the version of \eqref{main:eq2} without $R_{m}^{-1}\bm{j}$ and $R_{e}^{-1}\nabla\times \nabla\times \bm{u}$ (formally, $R_{e}=R_m=\infty$), as the {\it ideal MHD system}.

We consider the following boundary conditions for \eqref{sec:intro} on $\partial \Omega \times (0,T]$:  
\begin{eqnarray}\label{NS-bc}
\bu \times \bm{n} = \bm{0}, \quad P:=p + \frac{1}{2} |\bu|^2 = 0, \quad  \bm{B}\cdot \bm{n} = 0, \quad \mbox{ and } \quad 
\bm{E}\times \bm{n} = \bm{0}. 
\end{eqnarray}
where $\bm{n}$ is the unit outer normal vector. The initial conditions for the fluid velocity and the magnetic field are given for any {$\bm{x}\in\Omega$}
\begin{equation}
\bu(\bm{x},0) =  \bu_{0}(\bm{x}), \quad \bm{B}(\bm{x},0) =  \bm{B}_{0}(\bm{x}).
\end{equation}

In fact, the boundary conditions for $\bm{u}$ and $P$ in \eqref{NS-bc} can be seen as a vorticity boundary condition since $\bm{u}\times \bm{n}=0$ implies $(\nabla\times \bm{u})\cdot \bm{n}=0$ on $\partial \Omega$. We refer to  \cite{girault1990curl,hughes1987new} for similar boundary conditions for the Navier-Stokes equations. As we shall see below, these boundary conditions are the ones that lead to the helicity conservation of ideal MHD systems.

We briefly review some conserved quantities of \eqref{NS-bc} below. First of all, MHD equation \eqref{main:eq2} preserves the energy. 
\begin{theorem}
The MHD system \eqref{main:eq2} with the boundary condition \eqref{NS-bc} satisfies the following energy identity: 
 \begin{eqnarray}\label{resistive:energylaw} 
{1 \over 2}\frac{d}{dt} \|\bu\|_0^2 + \frac{\textsf{c}}{2}  \frac{d}{dt}\|\bm{B}\|_0^2 + R_{e}^{-1} \|\nabla\times \bu\|_0^2  + \textsf{c} R_m^{-1} \| \bm{j}\|_0^2
= (\bm{f}, \bu). 
\end{eqnarray}
\end{theorem}
The energy law \eqref{resistive:energylaw} is well known (see, for example, \cite{liu2009introduction}) and the key of the proof is a 
cancelation between the Lorentz force term $ ( \bm{j} \times \bm{B},\bu)=((\nabla\times \bm{B})\times \bm{B}, \bu)$ obtained by multiplying $\bu$ to the equation \eqref{main:eq2m} and the magnetic advection $ (\bu \times \bm{B}, \nabla \times \bm{B})$ obtained by multiplying $\bm{B}$ to \eqref{main:eqf}.  This cancelation reflects symmetry in the operator structure of the MHD system (c.f., \cite[(4.7)]{ma2016robust}). 

We now discuss the helicity conservation for \eqref{main:eq2}. There are two kinds of helicity in the MHD system: the magnetic helicity  $\mathcal{H}_{m}$ and the cross helicity $\mathcal{H}_c$, which are defined, respectively, as follows: 
\begin{equation*} 
\mathcal{H}_{m} := \int_{\Omega}\bm{B}\cdot\bm{A}\, {d}x, \quad \mbox{ and } \quad  \mathcal{H}_c :=  \int_{\Omega} \bm{B} \cdot \bm{u}\, {d}x.
\end{equation*} 
Here $\bm{A}$ is any potential such that $\nabla \times \bm{A} = \bm{B}$.  The definition of $\mathcal{H}_{m}$ is gauge invariant, i.e., not depending on the choice of the magnetic potential $\bm{A}$ since $\int_{\Omega}\bm{B}\cdot\nabla\phi\, {d}x=0$ for any scalar field $\phi$ with the given boundary conditions.

In ideal MHD systems, the magnetic and cross helicity is conserved. 
We state a slightly more general helicity identity as follows. The proofs (usually with vanishing boundary conditions and $R_{e}=R_{m}=\infty$) can be found in, e.g., \cite{maj2017mathematical}. 
\begin{lemma}\label{lem:continuous-helicity}
For the MHD system \eqref{main:eq2}, the following identity holds: 
\begin{equation} \label{id:Hm}
\frac{\mathrm{d}}{\mathrm{d}t} \mathcal{H }_m = -\int_{\partial\Omega}(\bm{A}_{t}+2\bm{E})\times\bm{A}\cdot \bm{n}\, {d}s-2R_{m}^{-1}\int_{\Omega}\bm{B}\cdot \nabla\times \bm{B}\, {d}x,
\end{equation} 
\begin{align}  \nonumber
\frac{d}{dt} \mathcal{H}_{c} = \int_{\partial \Omega} &( [\bu\times \bm{B}] \times \bu -  P\bm{B}-R_{m}^{-1}(\nabla\times \bm{B})\times \bm{u}-R_{e}^{-1}\bm{\omega}\times \bm{B})\cdot \bm{n}\, {d}s\\ &
+\int_{\Omega}\bm{f}\cdot \bm{B}\, {d}x -(R_{e}^{-1}+R_{m}^{-1})\int_{\Omega}\nabla\times \bm{B}\cdot \nabla\times \bm{u}\, {d}x.\label{id:Hc}
\end{align} 
\end{lemma}
Here the first term on the right hand side of \eqref{id:Hm} is due to the contribution of boundary terms and the second term reflects the effect of diffusion. With the boundary condition $\bm{B}\cdot\bm{n}=0$ as \eqref{NS-bc}, we can choose $\bm{A}$ such that $\bm{A}\times \bm{n}=0$ and thus the first term vanishes. In ideal MHD systems, the second term also vanishes and therefore the magnetic helicity is conserved.  Similarly, with the boundary conditions in \eqref{NS-bc} and $\bm{f}$ being any gradient field,  the right hand side of \eqref{id:Hc} vanishes in the ideal MHD limit and therefore the cross helicity is conserved.   This observation is summarized in the following theorem.
\begin{theorem}
In the ideal MHD systems with the boundary conditions \eqref{NS-bc} with $\bm{f}$ being a gradient field, the magnetic helicity and the cross helicity are conserved in the evolution:
$$
\frac{d}{d t} \mathcal{H}_m=\frac{d}{dt} \mathcal{H}_{c}=0.
$$
\end{theorem}
Note that the  magnetic helicity is conserved as long as (formally) $R_{m}^{-1}=0$, even if $R_{e}$ is finite.

\section{Helicity-preserving numerical discretization} \label{sec:alg}
In this section, we construct finite element methods that preserve both magnetic and cross helicity. 
To motivate the scheme, in Section \ref{helicity-polution} we show that extra terms pollute the helicity in some existing numerical schemes, e.g., the algorithms in \cite{hu2014stable}. 
We present our new finite element methods in Section \ref{sec:equiv} and prove their properties.

\subsection{Helicity-pollution in non-conservative schemes}\label{helicity-polution}

Let $\bm{V}^h \subset [H_0^1(\Omega)]^{3}$ and $Q^h \subset L^{2}_{0}(\Omega)$ be an inf-sup stable Stokes finite element pair and $\mathbb{Q}_h^V : [L^2(\Omega)]^3 \to \bm{V}^h$ be the $L^2$ projection. Define $\bm{Z}_{h}:=\bm{V}^h \times Q^h \times H_0^h({\rm div}, \Omega)\times H_{0}^{h}(\curl, \Omega)\times H_{0}^{h}(\curl, \Omega)$.  A seemingly natural discretization of \eqref{main:eq2}, in the spirit of existing discretization schemes, e.g., \cite{hu2014stable,hu2019structure}, is the following: Find $(\bm{u}_{h}(t), P_{h}(t), \bm{B}_h(t), \bm{E}_{h}(t), \bm{j}_{h}(t)) \in \bm{Z}_{h}$ such that for any $(\bm{v}_{h}, Q_{h}, \bm{C}_h, \bm{F}_{h}, \bm{k}_{h}) \in \bm{Z}_{h}$:
\begin{subeqnarray}\nonumber
(\partial_t \bu_{h}, \bm{v}_{h}) - (\bu\times (\nabla\times \bm{u}_{h}), \bm{v}_{h}) + R_{e}^{-1}(\nabla \times \bu_{h}, \nabla \times \bm{v}_{h})\\ - (P_{h}, \nabla\cdot \bm{v}_{h}) - \textsf{c} (\bm{j}_{h} \times \bm{B}_{h}, \bm{v}_{h})&=& (\bm{f}, \bm{v}_{h}),  \slabel{main:eq2m-h} \\ 
(\bm{j}_{h}, \bm{k}_{h}) - ( \bm{B}_{h}, \nabla\times \bm{k}_{h})  &=& \bm{0},\slabel{main:eqj-h} \\
(\partial_t \bm{B}_{h}, \bm{C}_{h}) + (\nabla \times \bm{E}_{h}, \bm{C}_{h}) &=& \bm{0}, \slabel{main:eqf-h} \\ 
R_m^{-1} (\bm{j}_{h}, \bm{F}_{h}) -   ( \bm{E}_{h} + {\bf{u}}_{h} \times \bm{B}_{h}, \bm{F}_{h} ) &=& \bm{0}, \slabel{main:eqo-h} \\ 
{({\bf{u}}_{h}, \nabla Q_{h})} &=& 0.
\end{subeqnarray}

From \eqref{main:eqj-h} and \eqref{main:eqo-h}, we have $\bm{j}_{h}=\nabla_{h}\times \bm{B}_{h}$ and $R_{m}^{-1}\bm{j}_{h}=\mathbb{Q}_{h}^{\curl}(\bm{E}_{h}+{\bf{u}}_{h} \times \bm{B}_{h})=\bm{E}_{h}+\mathbb{Q}_{h}^{\curl}({\bf{u}}_{h} \times \bm{B}_{h})$ (since $\bm{E}_{h}$ is already in $H^{h}_{0}(\curl, \Omega)$), respectively.  Therefore
\begin{equation}\label{eqn:EBu}
\bm{E}_{h}=R_{m}^{-1}\nabla_{h}\times \bm{B}_{h}-\mathbb{Q}_{h}^{\curl}({\bf{u}}_{h} \times \bm{B}_{h}).
\end{equation}
From \eqref{main:eqf-h}, $(\bm{B}_{h})_{t}+\nabla\times \bm{E}_{h}=0$, and $\frac{d}{dt}(\nabla\cdot \bm{B}_{h})=0$. 
Assume that $\bm{B}_{h}^{0}$, the initial data of $\bm{B}_{h}$, satisfies the divergence-free condition $\nabla\cdot \bm{B}_{h}^{0}=0$.  Then the divergence-free condition $\nabla\cdot \bm{B}_{h}=0$ holds at any time $t\geq 0$.  Therefore there exists $\bm{A}_{h}\in H^{h}_{0}(\curl)$ such that $\nabla\times \bm{A}_{h}=\bm{B}_{h}$. Then $\nabla\times ((\bm{A}_{h})_{t}+\bm{E}_{h})=0$ and without loss of generality ($\bm{A}_{h}$ is chosen up to a gradient potential), we have $(\bm{A}_{h})_{t}+\bm{E}_{h}=0$. Substituting $\bm{E}_{h}$ by \eqref{eqn:EBu}, we have 
$$
(\bm{A}_{h})_{t}=-R_{m}^{-1}\nabla_{h}\times \bm{B}_{h}+\mathbb{Q}_{h}^{\curl}({\bf{u}}_{h} \times \bm{B}_{h}).
$$
Consequently, by the Leibniz rule and integration by parts,
\begin{align*}
\frac{d}{dt}(\bm{A}_{h}, \bm{B}_{h})&=2((\bm{A}_{h})_{t}, \bm{B}_{h})=-2R_{m}^{-1}(\bm{B}_{h}, \nabla_{h}\times \bm{B}_{h})+2(\bm{u}_{h}\times \bm{B}_{h}, \mathbb{Q}_{h}^{\curl}\bm{B}_{h})\\
& =-2R_{m}^{-1}(\bm{B}_{h}, \nabla_{h}\times \bm{B}_{h})+2(\bm{u}_{h}\times \bm{B}_{h}, (\mathbb{Q}_{h}^{\curl}-\mathbb{I})\bm{B}_{h}). 
\end{align*}
For the second equality we have used the property \eqref{move-P}. 
Here the first term on the right hand side is due to the magnetic diffusion which is consistent with the continuous level \eqref{id:Hm}, while the second term is nonphysical due to the numerical scheme.

To see the pollution of the cross helicity, take $\bm{v}_{h}= \mathbb{Q}_{h}^{V}\bm{B}_{h}$ in \eqref{main:eq2m-h} and $\bm{C}_{h}=\mathbb{Q}_{h}^{\div}\bm{v}_{h}$ in \eqref{main:eqf-h}, where $\mathbb{Q}_{h}^{\div}: L^{2}\to H_{0}^{h}(\div, \Omega)$ is the $L^{2}$ projection to the finite element space:
\begin{align*}
\frac{d}{dt}(\bm{u}_{h}, \bm{B}_{h})&=((\bm{u}_{h})_{t}, \bm{B}_{h})+((\bm{B}_{h})_{t}, \bm{u}_{h})\\
&= (\bm{u}_{h}\times (\nabla\times \bm{u}_{h}), \mathbb{Q}_{h}^{V}\bm{B}_{h})-(\bm{u}_{h}\times ( \mathbb{Q}_{h}^{\curl} \nabla\times \bm{u}_{h}), \bm{B}_{h})
\\&\quad\quad +\textsf{c}((\nabla_{h}\times \bm{B}_{h})\times \bm{B}_{h}, (\mathbb{Q}_{h}^{V}-\mathbb{I})\bm{B}_{h})
-(P_{h}, \nabla\cdot \mathbb{Q}_{h}^{V}\bm{B}_{h})+(\bm{f}, \bm{B}_{h})
\\&\quad\quad -R_{m}^{-1}(\nabla_{h}\times \bm{B}_{h}, \nabla\times \bm{u}_{h})-R_{e}^{-1}(\nabla \bm{u}_{h}, \nabla \mathbb{Q}_{h}^{V}\bm{B}_{h})\\
&= (\bm{u}_{h}\times (\nabla\times \bm{u}_{h}), (\mathbb{Q}_{h}^{V}-\mathbb{I})\bm{B}_{h})+(\bm{u}_{h}\times ( (\mathbb{I}-\mathbb{Q}_{h}^{\curl}) \nabla\times \bm{u}_{h}), \bm{B}_{h})\\&\quad\quad+\textsf{c}((\nabla_{h}\times \bm{B}_{h})\times \bm{B}_{h}, (\mathbb{Q}_{h}^{V}-\mathbb{I})\bm{B}_{h})-(P_{h}, \nabla\cdot \mathbb{Q}_{h}^{V}\bm{B}_{h})+(\bm{f}, \bm{B}_{h})\\
&\quad\quad-R_{m}^{-1}(\nabla_{h}\times \bm{B}_{h}, \nabla\times \bm{u}_{h})-R_{e}^{-1}(\nabla \bm{u}_{h}, \nabla \mathbb{Q}_{h}^{V}\bm{B}_{h}),
\end{align*}
where the last three terms are due to source or diffusion terms as the continuous level, and the remaining terms are nonphysical due to the numerical discretization.

\subsection{Full-discrete finite element formulation} \label{sec:equiv}

We present a helicity-preserving full discretization for the MHD system. We use the Crank-Nicolson method as the temporal scheme. Similar conclusions also hold for semi-discretization with continuous time. We begin our discussion by defining 
\begin{equation}
\bm{X}_{h} =  [H^{h}_{0}({{\curl}},\Omega)]^5 \times H^{h}_{0}(\div,\Omega) \times H_0^h(\rm grad, \Omega).
\end{equation} 

Below we use $\bu_{h}, \bm{B}_{h}, \bm{\omega}_{h}, p_{h}, \bm{j}_{h}, \bm{H}_{h}, \bm{E}_{h}$ to denote the variables evaluated at the midpoint of the time interval $[t_n,t_{n+1}]$. For $\bu_{h}$ and $ \bm{B}_{h}$, this means
\begin{equation}
\bu_{h} := \frac{\bu_{h}^{n+1} + \bu_{h}^n}{2} \quad \mbox{ and } \quad  \bm{B}_{h} := \frac{\bm{B}_{h}^{n+1} + \bm{B}_{h}^n}{2}.  
\end{equation} 
For other variables whose time derivatives do not appear in the equations, one defines, e.g.,  $P_{h}:=P_{h}^{n+1/2}$ as an independent variable without referring to $P_{h}^{n}$ or $P_{h}^{n+1}$, i.e., one uses stagger grids in the time direction. 

We also denote
$$
D_{t}\int_{\Omega}\bm{a}\cdot\bm{b}\,{d}x:=\frac{1}{\Delta t}(\int_{\Omega}\bm{a}^{n+1}\cdot\bm{b}^{n+1}\,{d}x-\int_{\Omega}\bm{a}^{n}\cdot\bm{b}^{n}\,{d}x),
$$ 
as the difference of the inner product at two successive time steps.

 The main scheme can be written as follows. 
 \begin{algorithm}[H]
\caption{Main algorithm}
\label{alg:main}
\begin{algorithmic}
\State Given
$(\bm{u}^{0}, \bm{B}^{0})\in H_{0}^{h}(\curl, \Omega)\times H_{0}^{h}(\div, \Omega)$ and $\bm{f}\in L^{2}(\Omega)$,
\For{$n = 0$, $1$, $\cdots$, $N$}
\State 
Find {$(\bu_{h}^{n+1}, \bm{\omega}_{h}^{n+1/2}, \bm{j}_{h}^{n+1/2}, \bm{E}_{h}^{n+1/2}, \bm{H}_{h}^{n+1/2}, \bm{B}_{h}^{n+1}, P_{h}^{n+1/2}) \in \bm{X}_{h}$}, 
such that for all $(\bv_{h}, \bm{\mu}_{h}, \bm{k}_{h}, \bm{F}_{h}, \bm{G}_{h}, \bm{C}_{h}, Q_{h})  \in \bm{X}_{h}$:
\begin{subeqnarray}\label{main:eq3d}  \slabel{eqn1}
\left( {D_t} \bu_{h}, \bv_{h} \right) - (\bu_{h} \times \bm{\omega}_{h}, \bv_{h}) + R_e^{-1}(\nabla\times \bu_{h}, \nabla\times \bv_{h}) &&  \nonumber \\ + 
 (\nabla P_{h}, \bv_{h})  - \textsf{c} (\bm{j}_{h} \times \bm{H}_{h},\bv_{h} ) &=& (\bm{f}, \bv_{h}), \qquad  \\\slabel{eqn2}
\left( {D_t} \bm{B}_{h}, \bm{C}_{h} \right)  + (\nabla\times \bm{E}_{h}, \bm{C}_{h}) &=& 0, \\ \slabel{eqn3}
(R_m^{-1} \bm{j}_{h} - [\bm{E}_{h} + \bu_{h}\times \bm{H}_{h}], \bm{G}_{h})  &=& 0, \\\slabel{eqn4}
(\bm{\omega}_{h}, \bm{\mu}_{h}) - (\nabla\times\bu_{h},  \bm{\mu}_{h}) &=& 0, \\\slabel{eqn5}
(\bm{j}_{h}, \bm{k}_{h}) - (\bm{B}_{h}, \nabla\times \bm{k}_{h}) &=& 0, \\ \slabel{eqn6}
(\bm{B}_{h}, \bm{F}_{h}) - (\bm{H}_{h}, \bm{F}_{h}) &=& 0, \\ \slabel{eqn7}
(\bu_{h} ,\nabla Q_{h}) &=& 0, 
\end{subeqnarray} 
where 
\begin{equation} 
{D_t} \bu_{h} := \frac{\bu_{h}^{n+1} - \bu_{h}^n}{\Delta t} \quad \mbox{ and } \quad {D_t} \bm{B}_{h} := \frac{\bm{B}_{h}^{n+1} - \bm{B}_{h}^n}{\Delta t},
\end{equation} 
and 
$$
\bm{u}_{h}:=\frac{\bm{u}_{h}^{n+1}+\bm{u}_{h}^{n}}{2},  \quad \bm{B}_{h}:=\frac{\bm{B}_{h}^{n+1}+\bm{B}_{h}^{n}}{2},\quad \bm{\omega}_{h}:=\bm{\omega}_{h}^{n+1/2}, \quad P_{h}:=P_{h}^{n+1/2}, 
$$
$$
\bm{E}_{h}:=\bm{E}_{h}^{n+1/2}, \quad \bm{H}_{h}:=\bm{H}_{h}^{n+1/2}, \quad \bm{j}_{h}:=\bm{j}_{h}^{n+1/2}.
$$

\EndFor
\end{algorithmic}
\end{algorithm}

Figure 2 summarizes the choice of variables and spaces. 

\smallskip

\centerline{  \includegraphics[width=0.9\textwidth]{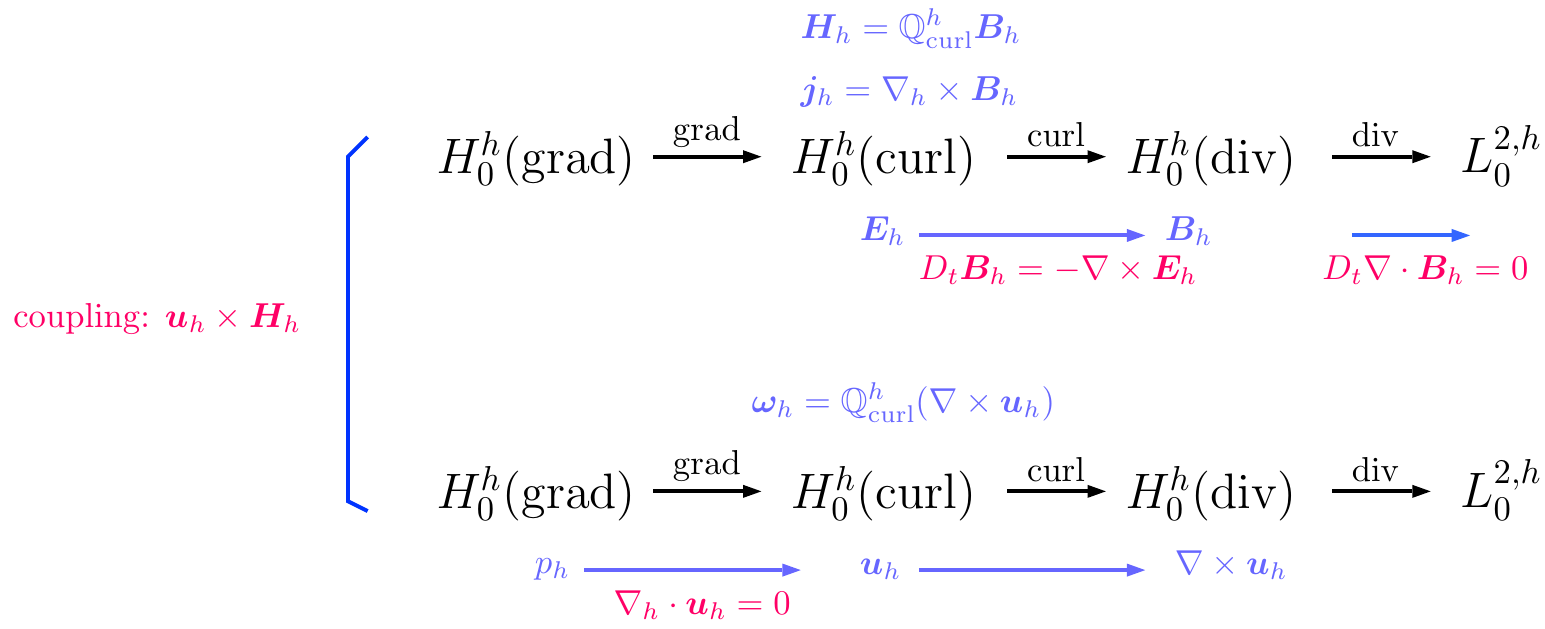} 
  }

\centerline{
  \includegraphics[width=0.9\textwidth]{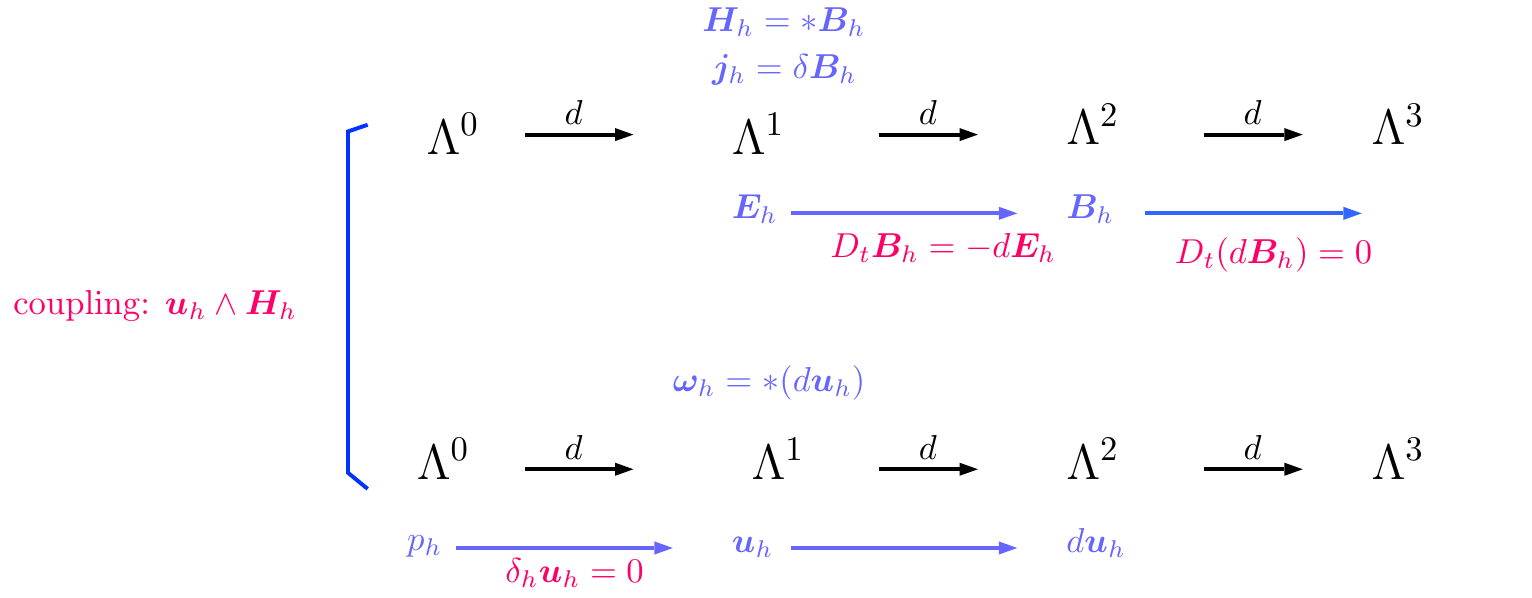} 
  }

\captionof{figure}{Choice of variables, spaces and their relations. Top: functions and differential operators. Bottom: differential form point of view. }
\smallskip

\begin{remark}
There are a number of variables in \eqref{main:eq3d}. However, several of them can be obtained easily as projections of other variables. For example, from \eqref{eqn3}-\eqref{eqn7}, we have
$$
\bm{E}_{h}=R_{m}^{-1}\nabla_{h}\times \bm{B}_{h}-\mathbb{Q}_{h}^{\curl}(\bm{u}_h\times \bm{H}_{h}), \quad \bm{\omega}_{h}=\mathbb{Q}_{h}^{\curl}(\nabla\times \bm{u}_{h}), 
$$
$$
\bm{j}_{h}=\nabla_{h}\times \bm{B}_{h}, \quad \bm{H}_{h}=\mathbb{Q}_{h}^{\curl}\bm{B}_{h}.
$$
In particular, $\bm{\omega}_{h}$ and $\bm{H}_{h}$ can be obtained by the local projections of $\nabla\times \bm{u}_{h}$ and $\bm{B}_{h}$, respectively.  Therefore in the numerical tests, one can solve the system \eqref{main:eq3d} efficiently. 
\end{remark}

In the analysis below, we will need the following basic fact about the Crank-Nicolson scheme.
\begin{lemma} 
For any given {$\bm{a}=(\bm{a}^{n})_{n=0, 1, \cdots, N} \subset [L^{2}(\Omega)]^{3}$}, we have 
\begin{eqnarray}
({D}_{t} \bm{a}, \bm{a}) &=& \frac{1}{2\Delta t} \left ( \|\bm{a}^{n+1}\|^2 - \|\bm{a}^n\|^2 \right ). 
\end{eqnarray}
Furthermore, for any pair of vectors {$(\bm{a}^{n}, \bm{b}^{n})_{n=0, 1, \cdots, N} \subset  [L^{2}(\Omega)]^{3} \times  [L^{2}(\Omega)]^{3}$}, we have the following identity: 
\begin{equation} 
{D_t} \int \bm{a} \cdot \bm{b} \, {d}x  = \int {D_t} \bm{a} \cdot \bm{b} \, {d}x + \int \bm{a} \cdot {D_t} \bm{b} \, {d}x ,   
\end{equation}  
where 
\begin{equation}
{D_t} \int \bm{a} \cdot \bm{b} \, {d}x := \frac{1}{\Delta t} \left (\int \bm{a}^{n+1} \cdot \bm{b}^{n+1}\, {d}x -  \int \bm{a}^{n} \cdot \bm{b}^{n} \, {d}x\right ).
\end{equation} 
Recall that 
$$
D_{t}\bm{a}:=\frac{\bm{a}^{n+1}-\bm{a}^{n}}{\Delta t}, \quad \bm{a}:=\frac{\bm{a}^{n+1}+\bm{a}^{n}}{2}, 
$$
and the same convention is used for $D_{t}\bm{b}$ and $\bm{b}$.
\end{lemma} 
The Gauss law $\nabla\cdot \bm{B}=0$ is automatically preserved in \eqref{main:eq3d}. Namely, we have
\begin{theorem}\label{divBd} 
If the initial data satisfies $\nabla\cdot\bm{B}_{h}^{0}=0$, then we have 
\begin{equation} 
\nabla\cdot\bm{B} _{h}^{n}= 0 \quad \forall n\geq 0. 
\end{equation} 
\end{theorem}
The proof is the same as in \cite{hu2014stable}. For completeness, we include the proof here.
\begin{proof} 
From the equation 
\begin{equation}
\left( {D_t} \bm{B}_{h}, \bm{C}_{h} \right)  + (\nabla\times \bm{E}_{h}, \bm{C}_{h}) = 0 \quad \forall \bm{C}_{h} \in H_0^h({\rm div}, \Omega).  
\end{equation} 
We have that ${D_t} \bm{B}_{h} = -\nabla \times \bm{E}_{h}$. Taking divergence, we obtain that
\begin{equation}
{D_t} \nabla\cdot \bm{B}_{h} = 0. 
\end{equation}  
This completes the proof. 
\end{proof} 

We now show the energy law for \eqref{main:eq3d}:  
\begin{theorem}
The discrete energy law holds: 
\begin{equation}\label{energy-eq}
\left [ ({D_t} \bu_{h}, \bu_{h}) + \textsf{c} ({D_t} \bm{B}_{h}, \bm{B}_{h} ) \right] + R_{e}^{-1} \|\nabla\times \bu_{h}\|^2 + \textsf{c} R_m^{-1} \|\bm{j}_{h}\|^2 = (\bm{f}, \bu_{h}),
\end{equation}
and
{
\begin{align}\nonumber
(\|\bm{u}_{h}^{n+1}\|^{2}+\text{c} &\|\bm{B}_{h}^{n+1}\|^{2})\leq \frac{1}{2}(\|\bm{u}_{h}^{0}\|^{2}+\text{c} \|\bm{B}_{h}^{0}\|^{2}) -\frac{1}{2}R_{e}^{-1}\sum_{j=0}^{n}(\Delta t)\|\nabla\times \bm{u}_{h}^{j+1/2}\|^{2}\\ \label{energy-ineq}
&-\textsf{c} R_{m}^{-1}\sum_{j=0}^{n}(\Delta t)\|\nabla_{h}\times \bm{B}_{h}^{j+1/2}\|^{2}+\frac{1}{2}{(\Delta t)}c_{p}^{2}R_{e}\sum_{j=0}^{n}\| \bm{f}^{j+1/2}\|^{2},
\end{align}}
where $c_{p}$ is the constant in the Poincar\'{e} inequality $\|\bm{u}_{h}\|\leq c_{p}\|\nabla\times \bm{u}_{h}\|$ for $\bm{u}_{h}$ satisfying $\nabla_{h}\cdot\bm{u}_{h}=0$.
\end{theorem}
\begin{proof}
Taking $\bv _{h}= \bu_{h}$ and $Q_{h}=P_{h}$ in the momentum equation of \eqref{main:eq3d}, we obtain 
\begin{eqnarray*} 
({D_t} \bu_{h}, \bu_{h})  + R_{e}^{-1} \|\nabla\times \bu_{h}\|^2 - \textsf{c} ([\nabla_h \times \bm{B}_{h}] \times \mathbb{Q}_h^{\rm curl} \bm{B}_{h}, \bu_{h}) = (\bm{f}, \bu_{h}). 
\end{eqnarray*} 
Moreover, we have that 
\begin{eqnarray*}
({D_t} \bm{B}_{h}, \bm{B}_{h}) &=& - (\nabla \times \bm{E}_{h}, \bm{B}_{h}) = - ( \nabla \times [R_m^{-1} \bm{j}_{h} - \mathbb{Q}_h^{\rm curl}(\bu_{h} \times \mathbb{Q}_h^{\rm curl} \bm{B}_{h})], \bm{B}_{h}) \\
&=&  - R_m^{-1} ( \bm{j}_{h}, \nabla_h \times \bm{B}_{h}) + (\bu_{h} \times \mathbb{Q}_h^{\rm curl} \bm{B}_{h}, \nabla_h \times \bm{B}_{h}) \\
&=&  - R_m^{-1} \| \bm{j}_{h}\|^2  + (\mathbb{Q}_h^{\rm curl}(\bu_{h} \times \mathbb{Q}_h^{\rm curl} \bm{B}_{h}), \bm{j}_{h})\\
&=&  - R_m^{-1} \| \bm{j}_{h}\|^2  + (\bu_{h} \times \mathbb{Q}_h^{\rm curl} \bm{B}_{h}, \bm{j}_{h}). 
\end{eqnarray*} 
Here for the last identity we used the fact that $\bm{j}_{h}\in H^{h}_{0}(\curl, \Omega)$.
 
Therefore, we have
\begin{eqnarray*}
- \textsf{c} (\bu_{h} \times \mathbb{Q}_h^{\rm curl} \bm{B}_{h}, \bm{j}_{h}) = \textsf{c} ( {D_t} \bm{B}_{h}, \bm{B}_{h})  + \textsf{c} R_m^{-1} \| \bm{j}_{h}\|_0^2.
\end{eqnarray*} 
This shows the equality \eqref{energy-eq}. Then \eqref{energy-ineq} follows from a sum and the estimate
\begin{align*}
\left|(\bm{f}, \bm{u}_{h})\right |\leq \|\bm{f}\|\|\bm{u}_{h}\|&\leq \frac{1}{2}R_{e}^{-1}c_{p}^{-2}\|\bm{u}_{h}\|^{2}+\frac{1}{2}c_{p}^{2}R_{e}\|\bm{f}\|^{2}\\
&\leq \frac{1}{2}R_{e}^{-1}\|\nabla\times \bm{u}_{h}\|^{2}+\frac{1}{2}c_{p}^{2}R_{e}\|\bm{f}\|^{2}.
\end{align*}
\end{proof}
\begin{remark}
We have another version of the energy estimates using a dual norm for the right hand side. Specifically, define 
$$
Z_{h}:=\{\bm{z}_{h}\in H_{0}^{h}(\curl, \Omega): \nabla_{h}\cdot \bm{z}_{h}=0\}.
$$
By the discrete Poincar\'e inequality, $\|\bm{z}_{h}\|_{Z_{h}}:=\|\nabla\times \bm{z}_{h}\|$ is a norm on $Z_h$. Then define the dual norm
$$
\|\bm f\|_{\ast}:=\sup_{\bm{v}_{h}\in Z_{h}}\frac{|\langle \bm f, \bm{v}_{h}\rangle|}{\|\bm{v}_{h}\|_{Z_{h}}}.
$$
By definition, $|\langle \bm f, \bm{v}_{h}\rangle|\leq \|\bm f\|_{\ast}\|\nabla\times\bm{v}_{h}\|$. Therefore we can remove the Poincar\'e constant in \eqref{energy-ineq} to get:
\begin{align*}
{\frac{1}{2}}(\|\bm{u}_{h}^{n+1}\|^{2}+\text{c} &\|\bm{B}_{h}^{n+1}\|^{2})\leq {\frac{1}{2}}(\|\bm{u}_{h}^{0}\|^{2}+\text{c} \|\bm{B}_{h}^{0}\|^{2}) -\frac{1}{2}R_{e}^{-1}\sum_{j=0}^{n}(\Delta t)\|\nabla\times \bm{u}_{h}^{j+1/2}\|^{2}\\ 
&-\textsf{c} R_{m}^{-1}\sum_{j=0}^{n}(\Delta t)\|\nabla_{h}\times \bm{B}_{h}^{j+1/2}\|^{2}+{(\Delta t)}R_{e}\sum_{j=0}^{n}\| \bm{f}^{j+1/2}\|_{\ast}^{2}.
\end{align*}
\end{remark}

We now discuss the magnetic and cross helicity for the discrete MHD system. The following theorems can be similarly stated and proved for  any contractible subdomain if the variables satisfy the  conditions \eqref{NS-bc} on  the boundary of the subdomain. Therefore we obtain identities for both local and global helicity. For simplicity of presentation, we focus on the helicity on $\Omega$, i.e., the global helicity.

\begin{theorem}\label{thm:dtHm}
For any solution of the discrete ideal MHD system \eqref{main:eq3d},  the following identity of the magnetic helicity holds: 
\begin{equation}\label{hmeq}  
{D_t} \int_{{\Omega}} \bm{B}_{h} \cdot \bm{A}_{h}\, dx  = R_{m}^{-1}\int_{{\Omega}}\bm{H}_{h}\cdot \bm{j}_{h}\, dx,
\end{equation} 
where $\bm{A}_{h}\in H_{0}^{h}(\curl, \Omega)$ is any vector potential of $\bm{B}_{h}$ satisfying $\nabla \times \bm{A}_{h}=\bm{B}_{h}$ in ${\Omega}$.
\end{theorem}
\begin{proof}
Since the magnetic Gauss law is precisely preserved, there exists $\bm{A}_{h}\in H^{h}_{0}({\curl}, {\Omega})$, such that $\nabla\times \bm{A}_{h}=\bm{B}_{h}$ in ${\Omega}$. 
Hence
\begin{eqnarray*} 
&&{D_t} \int_{{\Omega}} \bm{B}_{h}\cdot \bm{A}_{h} \, dx= \int_{{\Omega}} \nabla \times {D_t} \bm{A} _{h}\cdot \bm{A}_{h}\, dx + \int_{{\Omega}} {D_t} \bm{A} _{h}\cdot \bm{B}_{h} \, dx\\  
&=& \int_{{\Omega}} {D_t} \bm{A}_{h} \cdot \nabla \times \bm{A}_{h}\, dx + \int_{{\Omega}}  {D_t} \bm{A} _{h}\cdot \bm{B}_{h}\, dx = 2 \int_{{\Omega}} {D_t} \bm{A}_{h} \cdot \bm{B}_{h}\, dx. 
\end{eqnarray*} 
Note that ${D_t} \bm{B}_{h} = -\nabla \times \bm{E}_{h}$ and  ${D_t} \bm{B}_{h} = \nabla \times {D_t} \bm{A}_{h}$. Therefore, there exists $\phi\in H^{h}_{0}({{\grad}}, {\Omega})$ such that ${D_t} \bm{A}_{h} = -\bm{E}_{h} - \nabla \phi_{h}$. This means 
\begin{equation} 
 \int_{{\Omega}}  {D_t} \bm{A} _{h} \cdot  \bm{B} _{h}\, dx = - \int_{{\Omega}} (\bm{E}_{h} + \nabla \phi_{h})\cdot \bm{B}_{h}\, dx = - \int_{{\Omega}} \bm{E}_{h}\cdot  \bm{B}_{h}\, dx. 
\end{equation} 
However, $\bm{E}_{h} = - {\mathbb{Q}_h^{\curl}} ( \bu_{h} \times {\mathbb{Q}_h^{\curl}} \bm{B}_{h})+R_{m}^{-1}{\mathbb{Q}_h^{\curl}}\bm{j}_{h}$ by  \eqref{eqn3} and \eqref{eqn6}. Therefore, 
\begin{equation} 
(\bm{B}_{h},\bm{E}_{h}) = R_{m}^{-1}\left (\bm{B}_{h}, {\mathbb{Q}}_{h}^{\curl}\bm{j}_{h}\right )= R_{m}^{-1}\left ({\mathbb{Q}}_{h}^{\curl}\bm{B}_{h}, \bm{j}_{h}\right ). 
\end{equation}
This completes the proof. 
\end{proof}

We now show identities for the cross helicity. 
\begin{theorem}\label{thm:dtHc}
The following identity holds for the cross helicity:
\begin{equation}\label{hceq}  
{D_t} \int \bu_{h} \cdot \bm{B} _{h}\, dx=  -R_{e}^{-1}(\nabla\times \bm{u}_{h}, \nabla\times  \bm{H}_{h})-R_{m}^{-1} ( \nabla\times\bm{u}_{h},  \bm{j}_{h})+(\bm{f}, \bm{H}_{h}).
\end{equation} 
\end{theorem}
\begin{proof}
Taking $\bv_{h}={\mathbb{Q}_h^{\curl}} \bm{B}_{h}$, we have from  \eqref{eqn1}: 
\begin{subeqnarray}
\left( {D_t} \bm{u}_{h}, \bm{B}_{h} \right)+ ( {\mathbb{Q}_h^{\curl}} [ \nabla \times \bm{u}_{h}] \times \bm{u}_{h}, {\mathbb{Q}_h^{\curl}} \bm{B}_{h}) + (\nabla p_{h}, \bm{B}_{h}) \nonumber \\
+ R_{e}^{-1}(\nabla\times \bm{u}_{h}, \nabla\times {\mathbb{Q}_h^{\curl}} \bm{B}_{h})= (\bm{f}, {\mathbb{Q}_h^{\curl}} \bm{B}_{h}).  
\end{subeqnarray}
We also note  by  \eqref{eqn3} and \eqref{eqn6} that 
\begin{subeqnarray} 
\bm{E}_{h}  = -{\mathbb{Q}_h^{\curl}} [ \bu_{h} \times {\mathbb{Q}_h^{\curl}} \bm{B}_{h}]+ R_{m}^{-1}\bm{j}_{h}. 
\end{subeqnarray} 
On the other hand, we have that 
\begin{subeqnarray}
{D_t} \bm{B}_{h} = - \nabla \times \bm{E}_{h} = \nabla \times {\mathbb{Q}_h^{\curl}}  [ \bu_{h} \times {\mathbb{Q}_h^{\curl}} \bm{B}_{h}] - R_{m}^{-1}\nabla\times \bm{j}_{h}.
\end{subeqnarray} 
Consequently,
\begin{eqnarray*}
&&{D_t} \int_{{\Omega}} \bu_{h} \cdot \bm{B}_{h} \, dx= \left ( D_{t}\bu_{h}, \bm{B}_{h} \right ) + \left ( \bu_{h}, D_{t}\bm{B}_{h} \right )\\
&=& -( ({\mathbb{Q}_h^{\curl}}\nabla \times \bu_{h}) \times \bu_{h}, {\mathbb{Q}_h^{\curl}} \bm{B}_{h} ) -  (\nabla p_{h}, \bm{B}_{h}) -R_{e}^{-1}(\nabla\times \bm{u}_{h}, \nabla\times {\mathbb{Q}_h^{\curl}} \bm{B}_{h}) \\
&& + (\bu_{h},  \nabla \times {\mathbb{Q}_h^{\curl}}  [ \bu_{h} \times {\mathbb{Q}_h^{\curl}} \bm{B}_{h}])-R_{m}^{-1} ( \nabla\times\bm{u}_{h}, \bm{j}_{h}  )+(\bm{f}, {\mathbb{Q}_h^{\curl}} \bm{B}_{h})\\
&=&  -R_{e}^{-1}(\nabla\times \bm{u}_{h}, \nabla\times \bm{H}_{h})-R_{m}^{-1} (\bm{\omega}_{h},  \bm{j}_{h})+(\bm{f}, \bm{H}_{h}).
\end{eqnarray*}
\end{proof}
From Theorem \ref{thm:dtHm} and Theorem \ref{thm:dtHc}, we see that the discrete magnetic helicity and the discrete cross helicity are both conserved in the ideal MHD limit with suitable boundary conditions. We summarize this result as follows.
\begin{theorem}\label{thm:conservation}
Assume that $(\bm{f}, \bm{H}_{h})=0$. Then we have the helicity conservation in the ideal MHD limit:
$$
{D_t} \int_{{\Omega}} \bm{B}_{h} \cdot \bm{A}_{h}\, dx=0, \quad {D_t} \int_{{\Omega}} \bm{B}_{h} \cdot \bm{u}_{h}\, dx=0,
$$
i.e.,
$$
 \int_{{\Omega}} \bm{B}^{n}_{h} \cdot \bm{A}^{n}_{h}\, dx=\cdots=\int_{{\Omega}} \bm{B}^{0}_{h} \cdot \bm{A}^{0}_{h}\, dx,
\quad
   \int_{{\Omega}} \bm{B}^{n}_{h} \cdot \bm{u}^{n}_{h}\, dx=\cdots=\int_{{\Omega}} \bm{B}^{0}_{h} \cdot \bm{u}^{0}_{h}\, dx.
$$
\end{theorem}

The helicity provides a lower bound for  the energy \cite{arnold1999topological}. Thanks to the discrete de Rham complex and its properties and a judicious choice of the discrete formulation of the MHD equations, this bound can be carried over to the discrete level, supplying a control of the (local) discrete energy from below. We focus on the magnetic helicity, although a similar result holds for any divergence-free field.
\begin{proposition}\label{prop:lower-bound}
There exists a positive constant $C$ such that
\begin{equation} 
\mathcal{H}_{m}:=\int_{\Omega}\bm{B}_{h}\cdot \bm{A}_{h}\, {d}x\leq C \|\bm{B}_{h}\|^{2}.
\end{equation} 
\end{proposition} 
\begin{proof}
Choose $\bm{A}_{h}\in H^{h}_{0}(\curl, \Omega)$ such that $\nabla\times \bm{A}_{h}=\bm{B}_{h}$,
$(\bm{A}_{h}, \nabla \psi_{h})=0, ~\forall \psi_{h}\in H_{0}^{h}(\grad, {\Omega})$, and  $\bm{A}_{h}\times \bm{n}=0$ on $\partial {\Omega}$. 
By the discrete Poincar\'{e} inequality \cite{Arnold.D;Falk.R;Winther.R.2006a,Hiptmair.R.2002a}, there exists a universal positive constant $C$ such that $\|\bm{A}_{h}\|\leq C\|\nabla\times \bm{A}_{h}\|=C\|\bm{B}_{h}\|$. Consequently,
\begin{equation} 
\mathcal{H}_{m}=\int_{\Omega}\bm{A}_{h}\cdot \bm{B}_{h}\, dx\leq \|\bm{A}_{h}\|\|\bm{B}_{h}\| \leq C \|\bm{B}_{h}\|^{2}. 
\end{equation} 
\end{proof}
Note that the same argument works for variables at each time step, i.e., $\bm{B}_{h}^{n}$ and $\bm{A}_{h}^{n}$. Then we get estimates at each time step 
$$
\int_{\Omega}\bm{B}_{h}^{n}\cdot \bm{A}_{h}^{n}\, dx\leq C\|\bm{B}_{h}^{n}\|^{2}, \quad \forall n,
$$
in addition to the estimates at midpoints 
$$
\int_{\Omega}\bm{B}_{h}^{n+1/2}\cdot \bm{A}_{h}^{n+1/2}\, dx\leq C\|\bm{B}_{h}^{n+1/2}\|^{2}, \quad \forall n,
$$
which is stated in Proposition \ref{prop:lower-bound} by our  convention of notation.


\section{Numerical Experiments} \label{sec:num}

We report a couple of numerical tests on the convergence and the helicity conservation of the proposed scheme. In particular, we investigate and compare helicity changes with various Reynolds numbers in different algorithms. The implementation is based on the FEniCS project \cite{alnaes2015fenics}, and we choose the finite element spaces in the lowest order discrete de~Rham sequence (first order N\'ed\'elec and Raviart-Thomas elements etc.). 


\subsection{Convergence of the algorithm} 
In this section, we carry out a 3D convergence test with the following form of solutions on the domain $\Omega = (0,1)^3$. Let  
\begin{equation}
p = h(x) h(y) h(z),  
\end{equation} 
where $h(\mu) = (\mu^2 - \mu)^2$. Further, we let  
\begin{equation}
g_1(t) = 4 - 2t, \quad g_2(t) = 1+t \quad \mbox{ and } \quad g_3(t) = 1-t. 
\end{equation}
We now introduce analytic velocity and magnetic fields that satisfy the boundary conditions. Namely, 
\begin{eqnarray*}
\bm{u} = \left( \begin{array}{c} -g_1 h'(x) h(y) h(z) 
\\ - g_2 h(x) h'(y) h(z)   \\  -g_3 h(x) h(y) h'(z) \end{array} \right ) \quad \mbox{ and } \quad \bm{B} = \bm{\omega} = \nabla \times \bm{u}. 
 \end{eqnarray*}
With this setting, $\bm{u} \times \bm{n} = 0$, $\bm{B} \cdot \bm{n} = 0$ and the modified pressure $P = |\bm{u}|^2/2 + p$ satisfies the boundary condition. Furthermore, it holds that $\nabla \cdot \bm{B} = 0$.   { In the tests below, we  include nonzero source terms in Algorithm \ref{alg:main} due to the choice of the true solution.} 

Before presenting the convergence results, we first make some remarks on the solvers of the coupled system. As we shall see, the coupled system is easy to solve, even though it has more independent variables than existing schemes, e.g., those in \cite{hu2014stable}. 

In the tests below, we will solve the coupled system \eqref{main:eq3d} with an iterative process, referred to as the outer iteration. In each outer iteration, the first step is to solve $\bm{u}_{h}$, $P_{h}$ by treating other terms in the momentum equation explicitly, i.e., solving the following problem: find $({\bm{u}}_h, {P}_h) \in H_0^h({\rm curl},\Omega) \times H_0^h({\rm grad}, \Omega)$ for a given ${\bm{F}}$ and $g$, such that 
\begin{subeqnarray}\label{up}
(\Delta t)^{-1}({\bm{u}}_h, \bm{v}_h)  + (\nabla {P}_h, \bm{v}_h) = ({\bm{F}}, \bm{v}_h), \quad \forall \bm{v}_h \in H^h_0({\rm curl}, \Omega) \\ 
({\bm{u}}, \nabla Q_h) = (g, Q_h), \quad \forall Q_h \in H^h_0({\rm grad},\Omega). 
\end{subeqnarray}
If $\bm{F}=0$, the above system boils down to a Poisson equation for $P$. We solve \eqref{up} by  an AMG-preconditioned minimum residual iterative method. Table \ref{tab:itnum} shows the uniform convergence with respect to the mesh size. After obtaining $\bm{u}_{h}$ and $P_{h}$ from solving \eqref{up}, we update other variables in \eqref{main:eq3d} by simple operations. For example, $\bm{\omega}_{h}$ and $\bm{E}_{h}$ are updated from \eqref{eqn4} and \eqref{eqn3} by the $L^{2}$ projections of $\nabla\times\bm{u}_{h}$ and $-\bm{u}_{h}\times \bm{B}_{h}$, respectively.

 The outer iterations typically takes about 4 to 5 iterations to achieve the appropriate tolerance, e.g., the difference of the $L^2$ norms between two consecutive iterations divided by the time step size is smaller than $10^{-7}$. As $R_e$ and $R_m$ become smaller, the convergence takes more nonlinear (outer) iterations. 
\begin{table}[h!!]
\centering
{\small{\begin{tabular}{|c||c|}
\hline
Mesh size &  {\textsf{Iteration Numbers}}  \\ 
\hline   
$h_x = h_y = h_z =  2^{-2}$   &  11 \\
\hline 
$h_x = h_y = h_z =  2^{-3}$   &  11  \\
\hline  
$h_x = h_y = h_z = 2^{-4}$   & 13 \\ 
\hline 
$h_x = h_y = h_z = 2^{-5}$   & 13  \\
\hline 
$h_x = h_y = h_z = 2^{-6}$   & 13  \\ 
\hline
\end{tabular}
\vspace{+3mm}\caption{Number of iterations of preconditioned MINRES to achieve relative tolerance $10^{-10}$ for solving \eqref{up}, $\Delta t=1$.
}\label{tab:itnum}}}
\end{table}

The convergence results are shown in Table \ref{table:convergence}.

\begin{table}
\begin{center}
{\small{\begin{tabular}{|c|c|c|c|c|c|c|}
\cline{1-7}
$h$  & $\|\bm{B} - \bm{B}_h\|_0$  & order & $\|\bm{u} - \bm{u}_h\|_{0}$ & order &$\|p - p_h\|_1$ & order 
\\ \cline{1-7} 
$2^{-2}$   & 1.60E-3 &  x      & 4.15E-4   &    x        & 2.15E-4 & x 
\\ \cline{1-7} 
$2^{-3}$   & 7.80E-4 & 1.04  & 2.18E-4   &   0.93    & 1.24E-4 & 0.79 
 \\ 
\cline{1-7} 
$2^{-4}$ & 3.40E-4 & 1.20 &  1.05E-4   &   1.05    &  6.44E-5 & 0.95 
\\ 
\cline{1-7} 
$2^{-5}$ & 1.63E-4 & 1.06  &  5.30E-5  &   0.99    &  3.25E-5 & 0.99
 \\ 
\cline{1-7} 
\end{tabular}\vspace{+3mm}\caption{Convergence results for the MHD system. The error is computed at the time level $T = 1$ with the Crank-Nicolson time stepping with $\Delta t = 0.01$. $R_e = R_m = 10^4$.} \label{table:convergence}}}
\end{center}
\end{table}


\subsection{Tests for Helicity conservation} 

In this section, we investigate the helicity behavior of our algorithms with various Reynolds numbers. We also compare the algorithm to another discretization based on existing schemes \cite{schotzau2004mixed}.

In the tests below, we use the following initial conditions for $\bm{u}_{h} ^{0}= (u_1, u_2, u_3)^T$: 
\begin{eqnarray*} 
u_1 &=& -\sin (\pi (x - 1/2)) \cos ( \pi (y - 1/2)) z(z-1)   \\ 
u_2 &=& \cos (\pi (x - 1/2)) \sin (\pi (y - 1/2)) z(z-1) \quad \mbox{ and } u_3 = 0. 
\end{eqnarray*}
For the magnetic field, we provide the following initial condition: 
\begin{eqnarray*}
\bm{B}_{h}^{0} = (-\sin (\pi x) \cos ( \pi y), \cos (\pi x ) \sin (\pi y ), 0)^T. 
\end{eqnarray*}
Figure \ref{init} shows the initial conditions for $\bm{u}_{h}$ and $\bm{B}_{h}$. We note that the desired boundary conditions are satisfied:  
\begin{equation} 
\bm{u}_{h}^{0} \times \bm{n} = 0 \quad \mbox{ and } \quad \bm{B}^{0}_{h} \cdot \bm{n} = 0 \mbox{ on } \partial \Omega. 
\end{equation} 
\begin{figure}[h]
\centering 
\includegraphics[width=6cm, height=6cm]{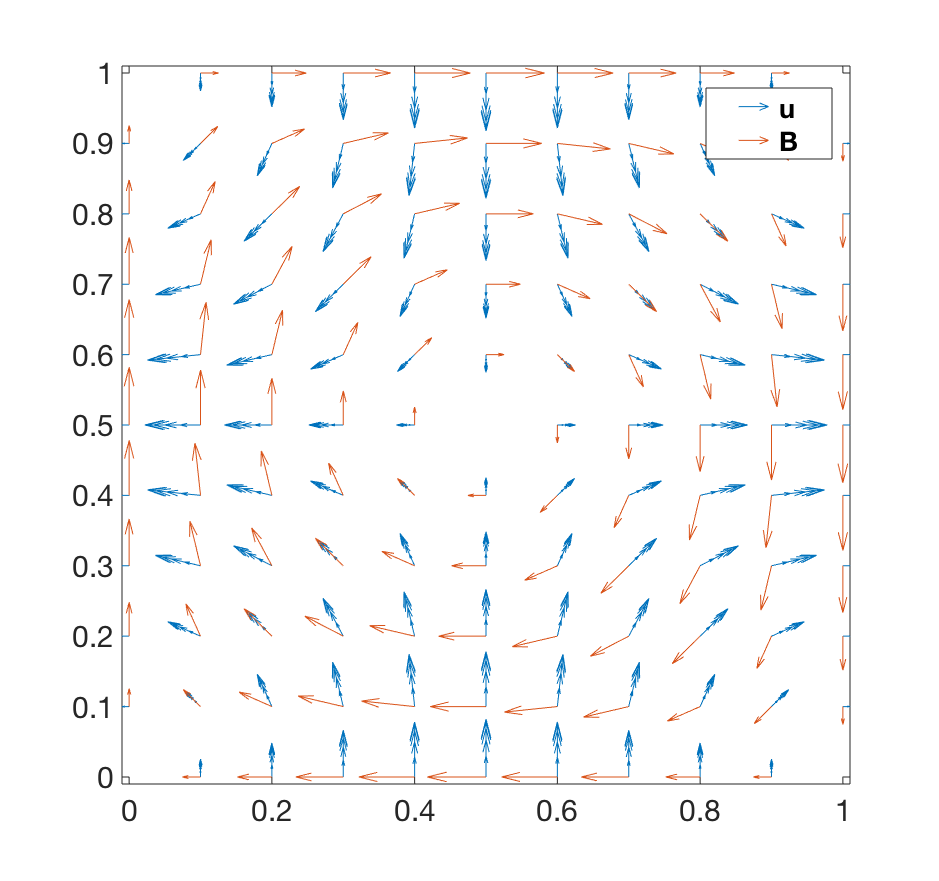}
\caption{Top view, or projection on the $xy$-plane of initial data $\bm{u}_{h}^{0}$ and $\bm{B}_{h}^{0}$}\label{init}  
\end{figure}

Furthermore, we have $\nabla \cdot \bm{B}^{0}_{h} = 0$ for the initial data and the helicity vanishes, i.e., $\mathcal{H}_m = \mathcal{H}_c = 0$.

To evaluate the magnetic helicity of our algorithm, we obtain $\bm{B}_{h}$ and compute the potential $\bm{A}_h$ by solving the following equation: find $\bm{A}_h \in H_0^h(\curl, \Omega)$ such that 
\begin{equation}\label{A-eqn}
(\nabla \times \bm{A}_h, \nabla \times \bm{C}_h) = (\bm{B}_h, \nabla \times \bm{C}_h), \quad \forall \bm{C}_h \in H_0^h(\curl, \Omega). 
\end{equation}
Since $\curl$ has a nontrivial kernel, \eqref{A-eqn} is a singular system. However, this non-uniqueness does not affect the helicity. In the implementation, we apply the Krylov space method, i.e., GMRES with ILU preconditioners, to solve \eqref{A-eqn}, which is known to converge for consistent singular problems \cite{ipsen1998idea,lee2008sharp}.

\begin{figure}[h]
\centering 
\includegraphics[width=6cm, height=6cm]{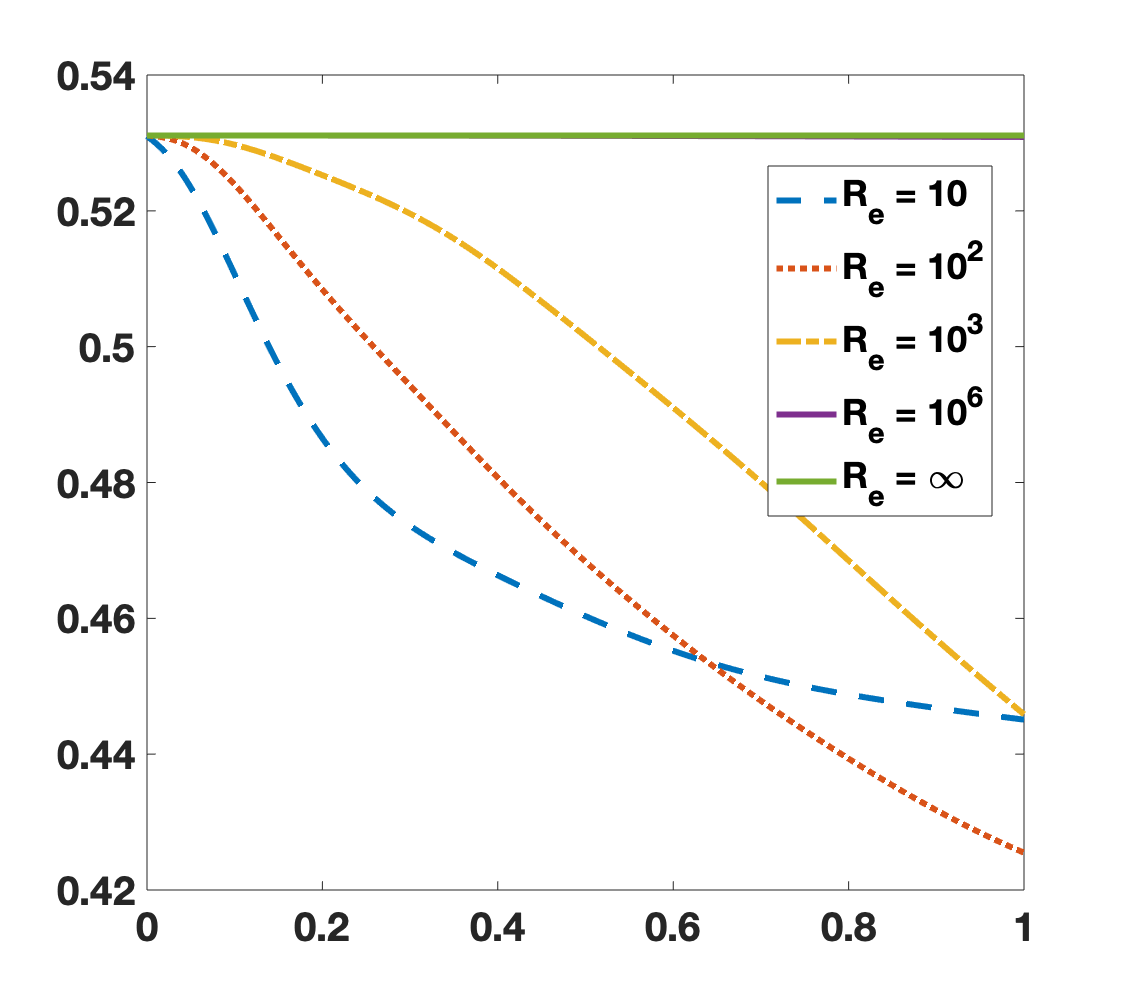}
\includegraphics[width=6cm, height=6cm]{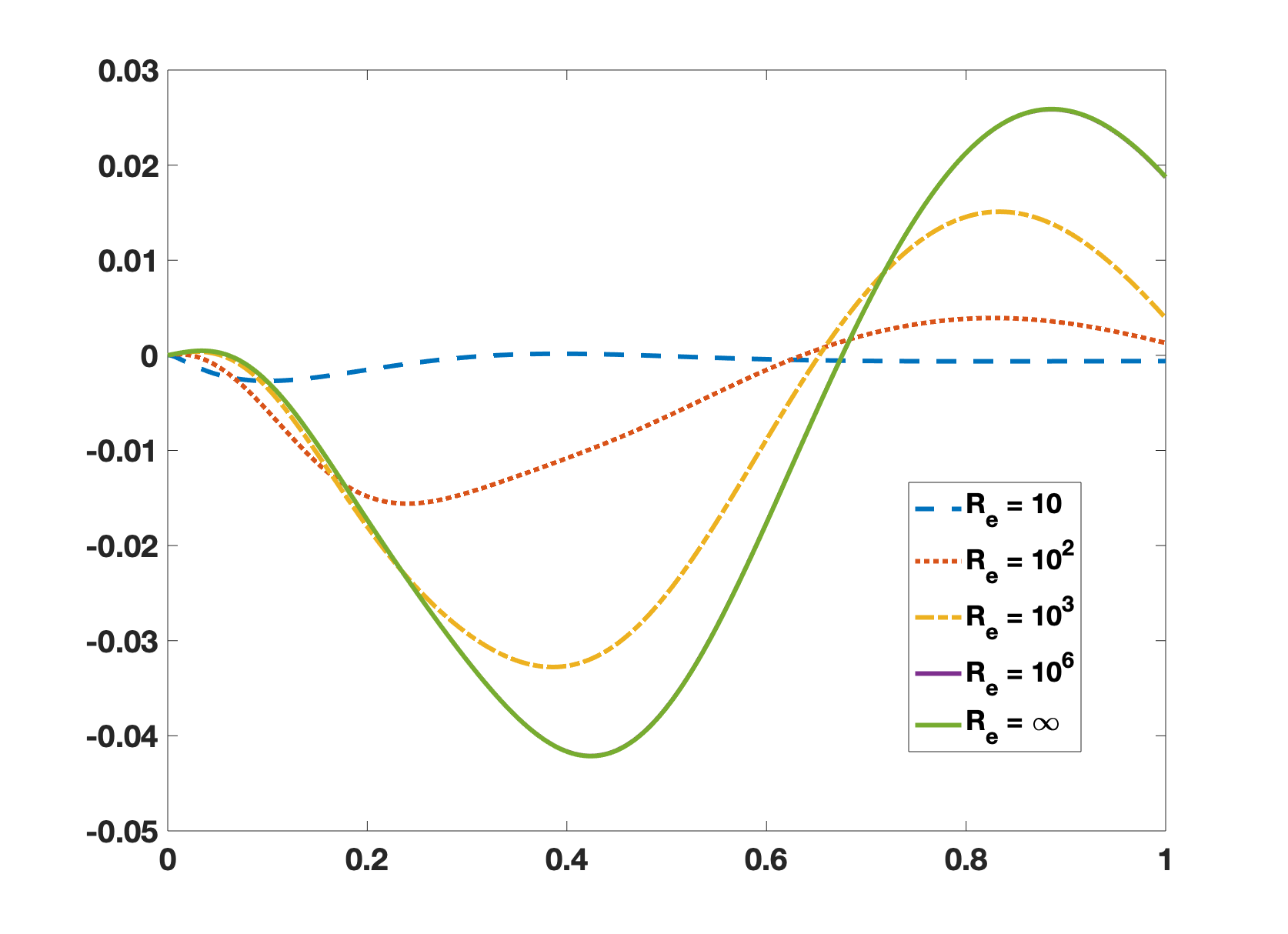}
\caption{(Left) Energy and (Right) $\int_\Omega \left( \nabla \cdot \mathbb{Q}_h^{\div} \bu_h \right )\, dx$ plot as a function of time with different Reynolds numbers. Note that $h = 1/8, \Delta t= 1/1000$ and $R_m = 10^7$.}\label{ediv} 
\end{figure}

\begin{figure}[h]
\centering 
\includegraphics[width=14cm, height=6cm]{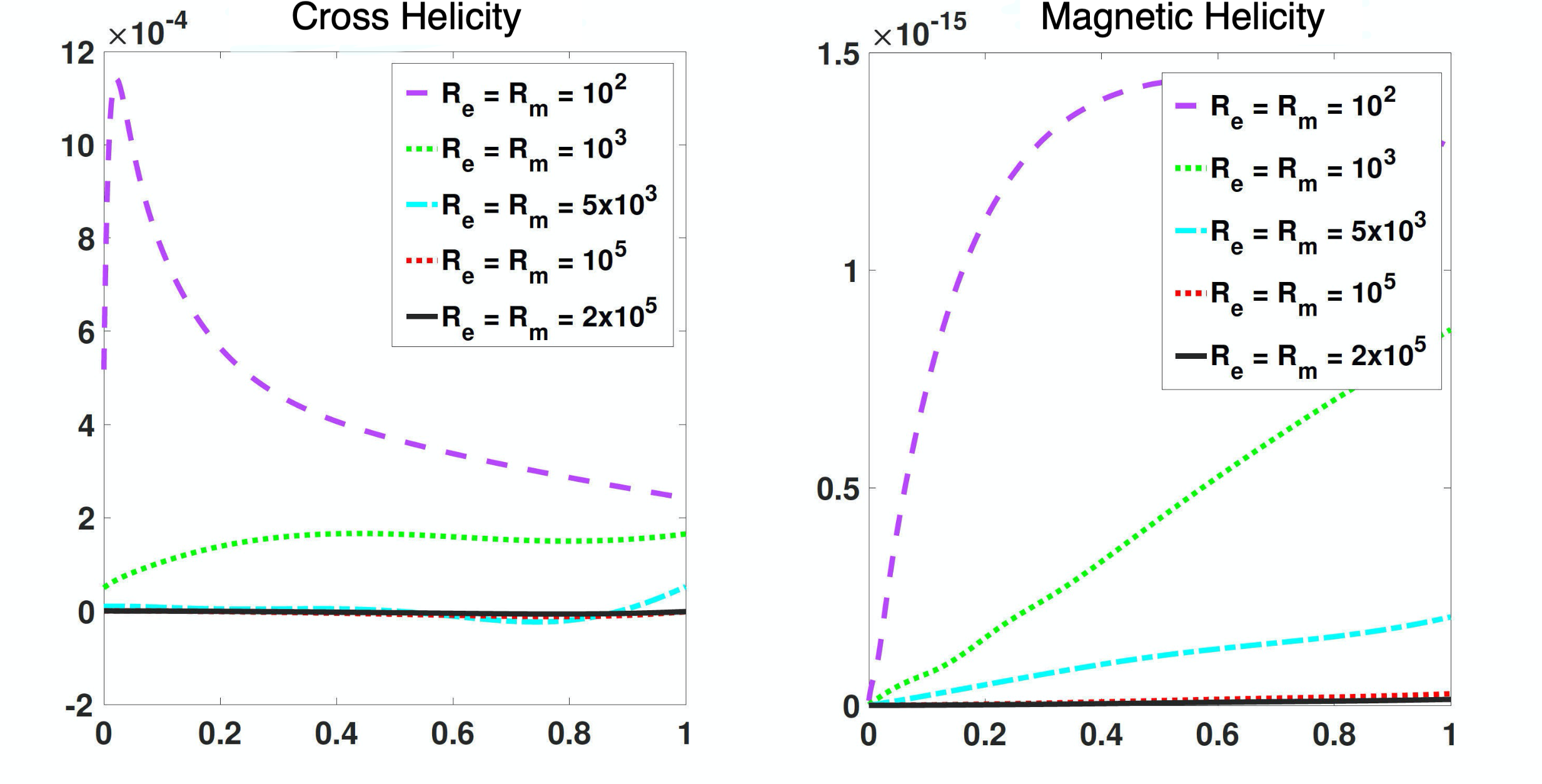}
\caption{Changes in the right hand side of the equation for $\mathcal{H}_c$ (left) \eqref{hceq} and $\mathcal{H}_m$ (right) \eqref{hmeq} obtained by Algorithm \ref{alg:main} with various choices of $R_{e}$ and $R_{m}$. Here $h = 1/16$ and $\Delta t = 1/1000$.} 
\label{fig:polution}
\end{figure}

\begin{figure}[h]
\centering 
\includegraphics[width=16cm, height=6cm]{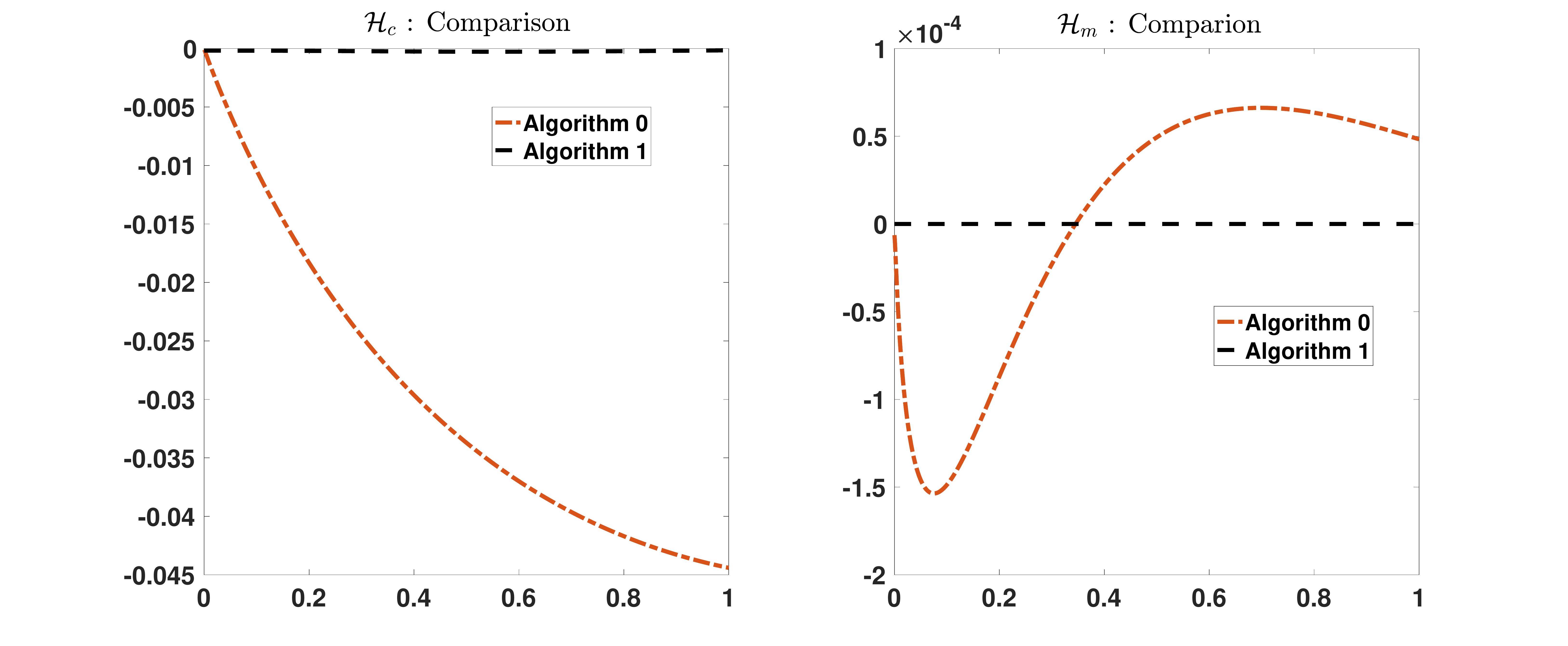}
\caption{$\mathcal{H}_c$ (left) and $\mathcal{H}_m$ (right) from Algorithm \ref{alg:0} and Algorithm \ref{alg:main}, respectively, with $R_e, R_m = 5000$, ${c} = 0.01$, $h = 1/16$ and $\Delta t = 1/1000$.} 
\label{fig:algorithm01}
\end{figure}

We now discuss the effect of resistivity  on the cross and magnetic helicity. Figure \ref{fig:polution} shows the evolution of helicity in Algorithm \ref{alg:main} with various choices of $R_{e}$ and $R_{m}$.  As $R_{e}$ and $R_{m}$ increase, the helicity is closer to be conserved. This is consistent with Theorem \ref{thm:conservation} stating that both the magnetic and the cross helicity are conserved in the ideal MHD limit.

To compare the helicity from Algorithm \ref{alg:main} and other algorithms, we consider another finite element algorithm based on the scheme proposed in \cite{schotzau2004mixed} for solving the stationary incompressible MHD system. 

The finite element scheme presented in \cite{schotzau2004mixed} has $\bm{B}_{h}$ in the N\'ed\'elec space with any stable Stokes pair for $\bm{u}_{h}$ and $P_{h}$. To show the effect of the discretization for the magnetic part and adapt the scheme to the vorticity boundary conditions \eqref{NS-bc},  we shall only use the scheme for the magnetic part of the algorithm and consider the time-dependent setting.
\begin{customthm}{0}\label{alg:0}
Find $(\bm{u}_h,\bm{\omega}_h,\bm{B}_h, P_h) \in H^h_0({\rm curl}, \Omega) \times H^h_0({\rm curl}, \Omega) \times H^h_0({\rm curl}, \Omega) \times H^h_0({\rm grad}, \Omega)$ such that for all $(\bm{v}_h,\bm{\mu}_h,\bm{C}_h, Q_h) \in H^h_0({\rm curl}, \Omega) \times H^h_0({\rm curl}, \Omega) \times H^h_0({\rm curl}, \Omega) \times H^h_0({\rm grad}, \Omega)$,
\begin{subeqnarray}\label{variation:schotzau} \nonumber
\left( {D_t} \bm{u}_{h}, \bm{v}_{h} \right) - (\bm{u}_{h} \times \bm{\omega}_{h}, \bm{v}_{h}) + R_e^{-1}(\nabla\times \bm{u}_{h}, \nabla\times \bm{v}_{h})  && \\ 
+ (\nabla P_{h}, \bm{v}_{h})  - {c} 
( (\nabla \times \bm{B}{_h}) \times \bm{B}_{h},\bm{v}_{h} ) &=& 0, \qquad  \\
 (\bm{\omega}_{h}, \bm{\mu}_{h}) - (\nabla\times\bm{u}_{h},  \bm{\mu}_{h}) &=& 0, \\ \slabel{sch-2}
\left( {D_t} \bm{B}_{h}, \bm{C}_{h} \right)  -\left ( \bm{u}_h \times \bm{B}_h, \nabla\times\bm{C}_h \right )+R_{m}^{-1}\left ( \nabla\times \bm{B}_h, \nabla\times \bm{C}_h \right ) &=& 0,\\ 
(\bm{u}_{h} ,\nabla Q_{h}) &=& 0.
\end{subeqnarray} 
\end{customthm}
In Algorithm \ref{alg:0}, we use the Crank-Nicolson time stepping as Algorithm \ref{alg:main}. In \cite{schotzau2004mixed} there is a Lagrange multiplier to impose the weak divergence-free condition for the magnetic field, i.e.,
\begin{equation}\label{weak-div}
(\bm{B}_h, \nabla z_h)=0, \quad\forall z_h \in H_0^{h}(\grad, \Omega).
\end{equation}
However, we may drop this constraint in the above time dependent formulation because we conclude $\left( {D_t} \bm{B}_{h}, \nabla z_{h} \right)=0, ~\forall z_{h}\in H_{0}^{h}(\grad, \Omega)$ by taking $\bm{C_{h}}=\nabla z_{h}$ in \eqref{sch-2}, i.e., if the initial data satisfies \eqref{weak-div}, then the solution satisfies \eqref{weak-div} at any time step.


Figure \ref{fig:algorithm01} compares the cross and magnetic helicity produced in Algorithm \ref{alg:0} and Algorithm \ref{alg:main}, respectively. In fact, for Algorithm \ref{alg:0}, we do not even have a precise definition of the magnetic helicity since the discrete magnetic field is not divergence-free. The curve in Figure \ref{fig:algorithm01}  for $\mathcal{H}_{m}$ demonstrates a discrete helicity computed by projecting the magnetic field to the divergence-free Raviart-Thomas space.

Figure \ref{fig:algorithm01} shows that even for the resistive MHD systems, our asymptotic-helicity-conservative scheme Algorithm \ref{alg:main} shows a significant difference in the helicity behavior compared to Algorithm \ref{alg:0}  which is not designed with an emphasis on the helicity-preservation.  

{
Figure \ref{fig:vel-2} shows the helicity from computation with various fluid Reynolds numbers and a large magnetic Reynolds number. The cross helicity varies in the time evolution, while the magnetic helicity nearly remains constant. This agrees with the fact that the evolution of the cross helicity depends on $R_{e}$ and $R_{m}$, while the magnetic helicity only depends on $R_{m}$ (see Lemma \ref{lem:continuous-helicity}). Figure \ref{fig:vel-2} also shows the helicity in the ideal MHD system (formally $R_{e}=R_{m}=\infty$), where both the cross helicity and the magnetic helicity are conserved. 
}

{We also plot the snapshot of the velocity and magnetic fields in the time evolution in Figure \ref{fig:vel}  to verify the stability of our computation.} 

\begin{figure}[h]
\centering 
\includegraphics[width=4.1cm, height=5cm]{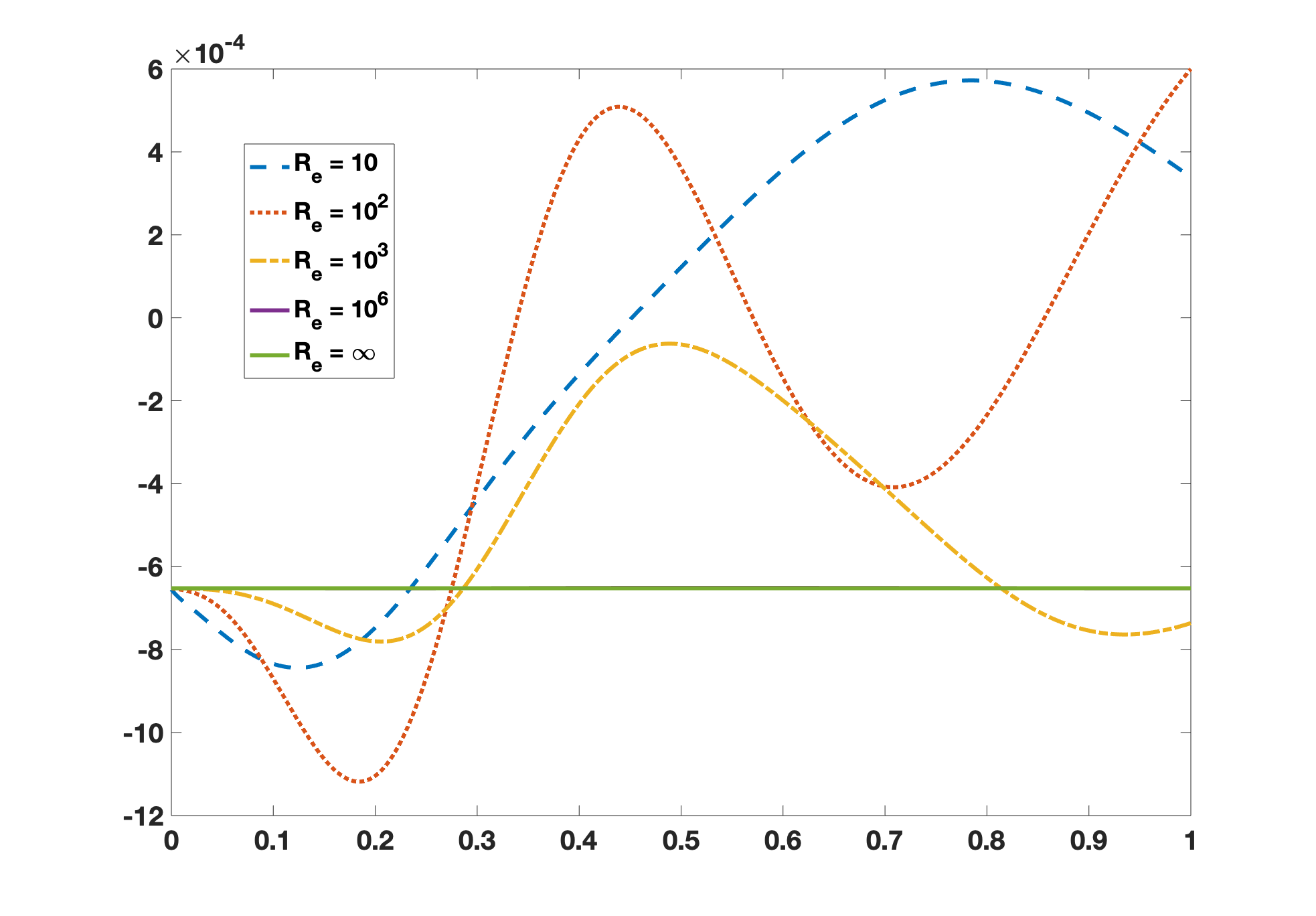} 
\includegraphics[width=4.1cm, height=5cm]{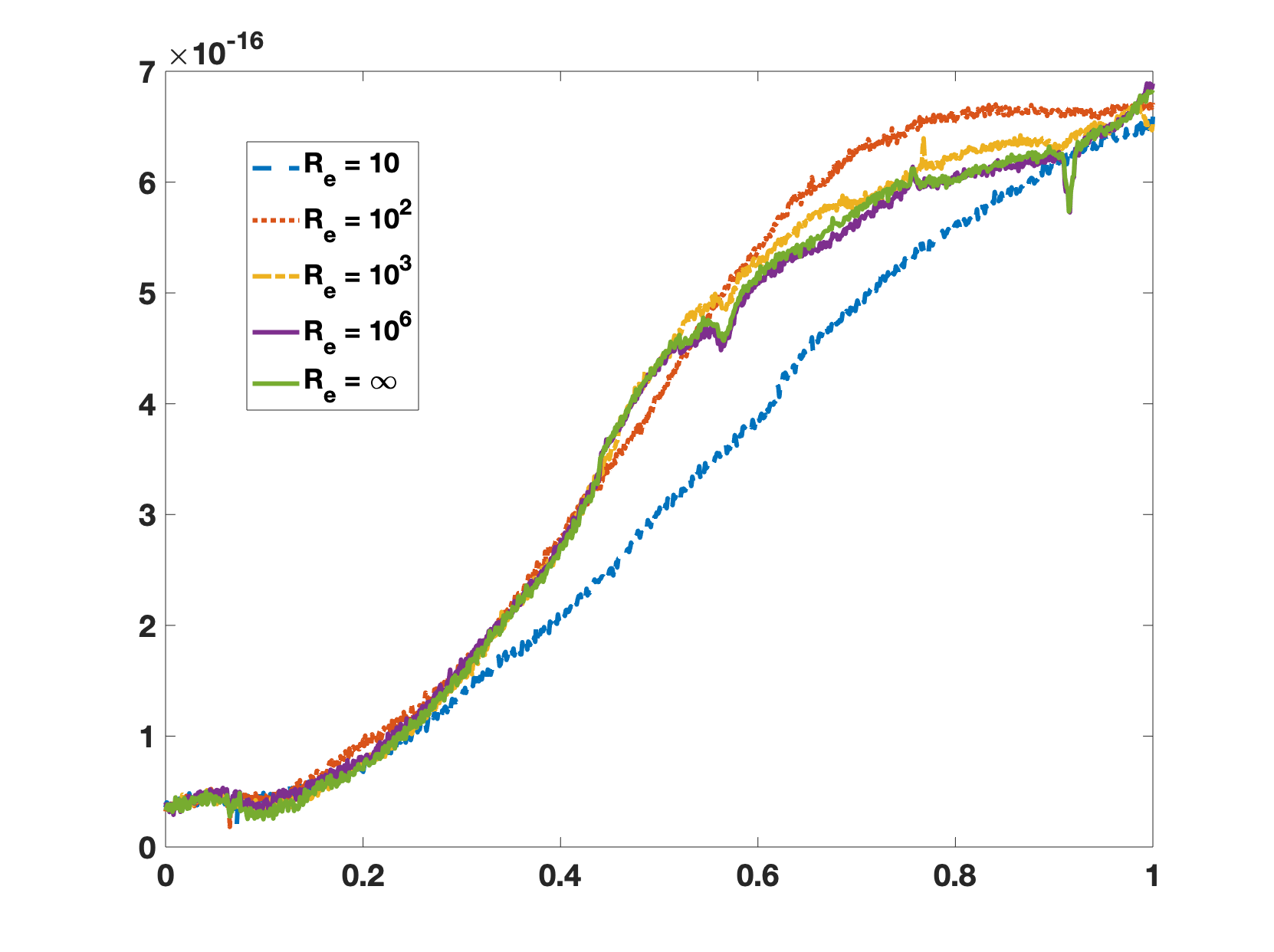} 
\includegraphics[width=4.1cm, height=5cm]{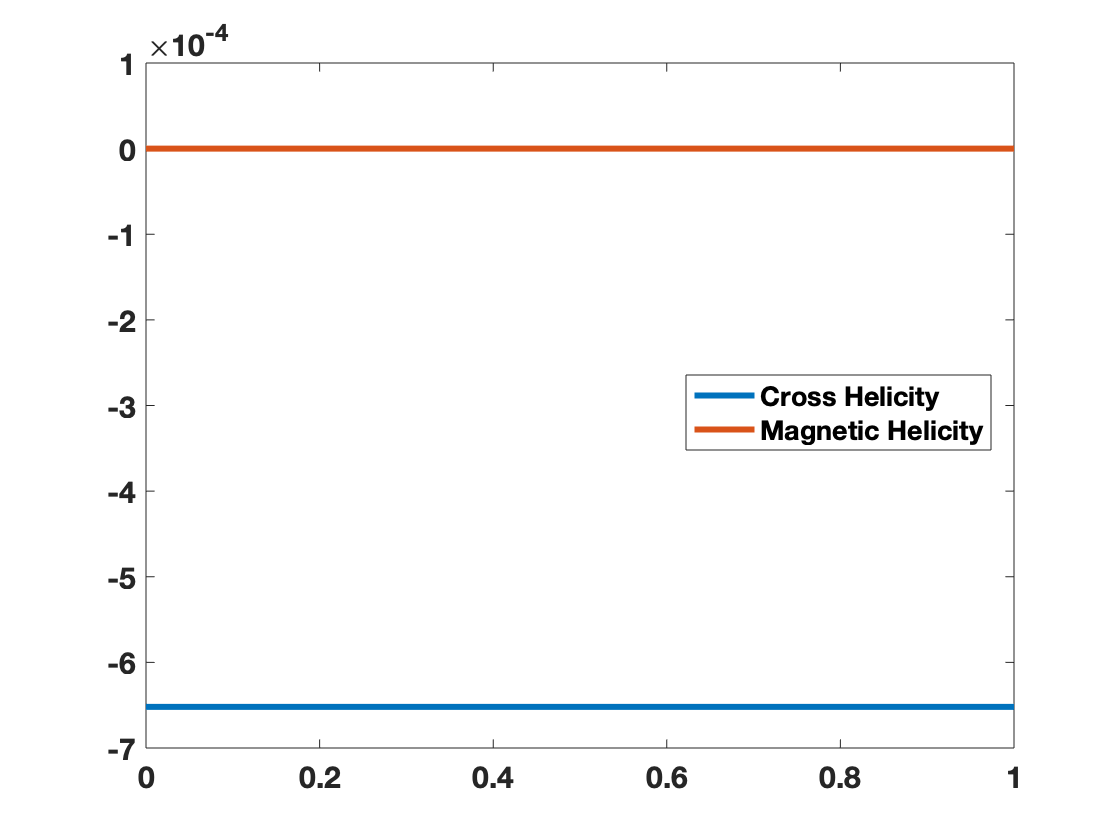} 
\caption{{Plot of the cross helicity (left) and the magnetic helicity (middle) with $R_m = 10^7$  and various finite values of  $R_e$, and (right) Plot of cross and magnetic helicity for the ideal case ($R_m = R_e = \infty$) obtained from Algorithm \ref{alg:main} with ${c} = 1, h = 1/8$ and $\Delta t = 1/1000$. }} 
\label{fig:vel-2}
\end{figure}

\begin{figure}[h]
\centering 
\includegraphics[width=3.2cm, height=3.2cm]{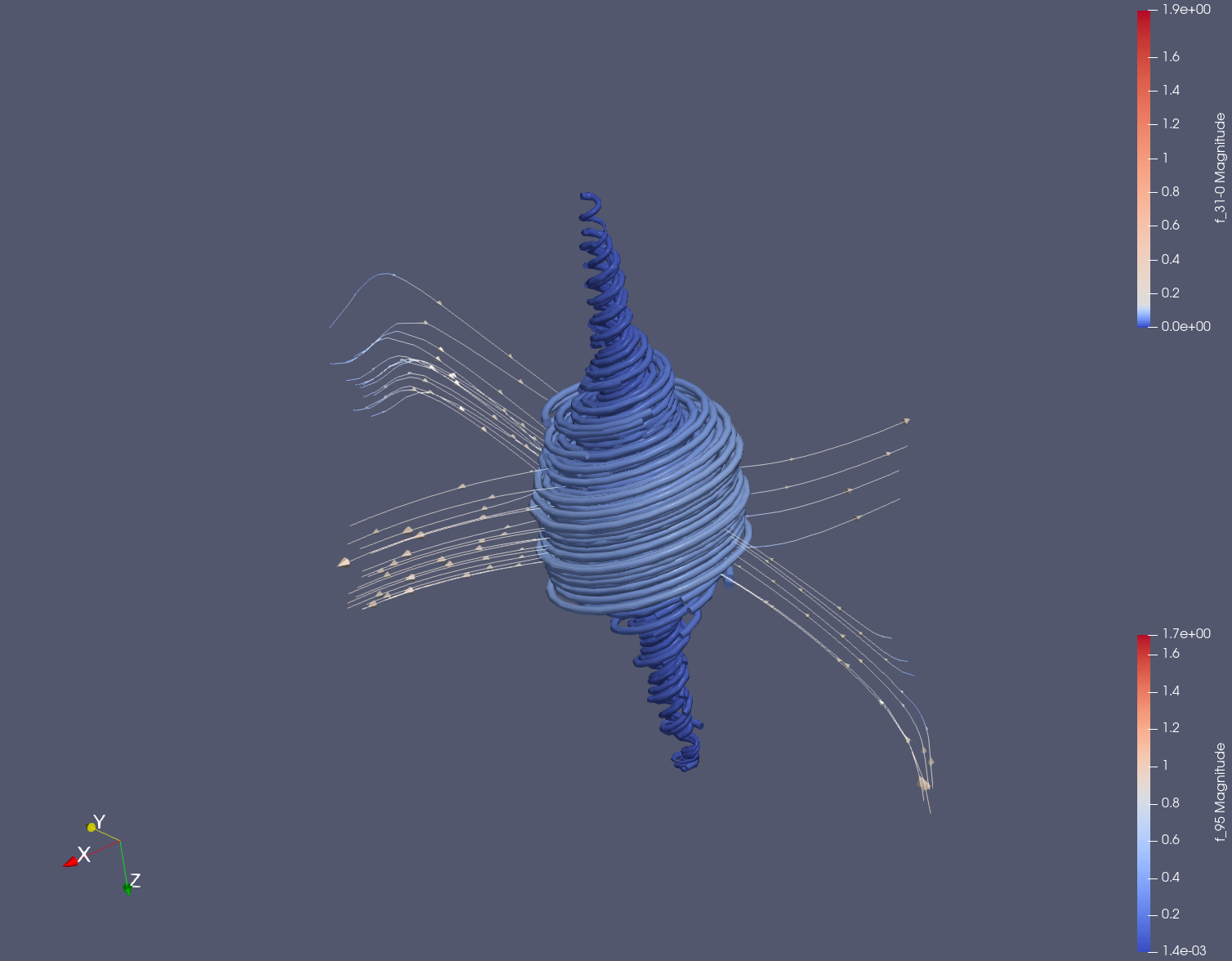} 
\includegraphics[width=3.2cm, height=3.2cm]{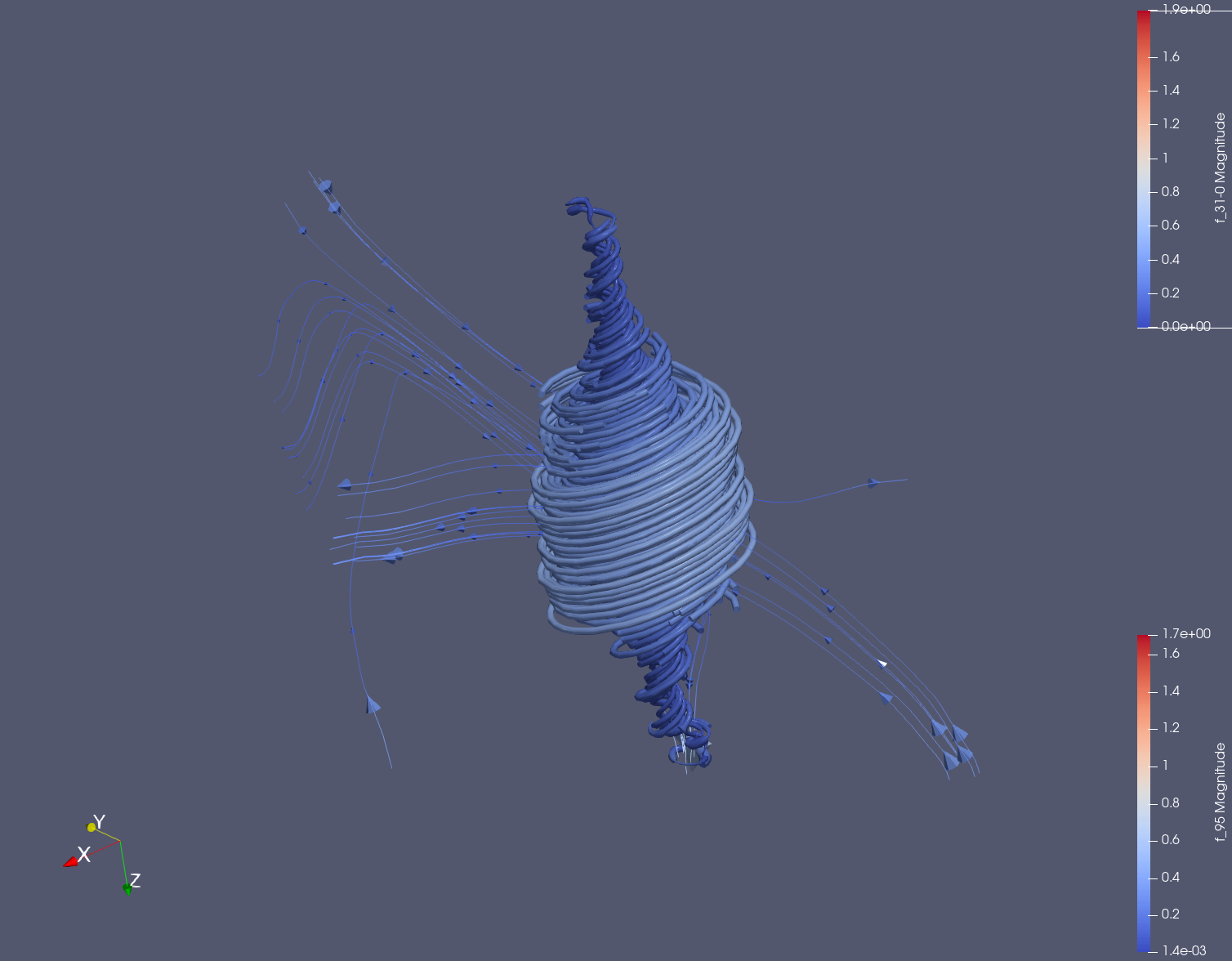} 
\includegraphics[width=3.2cm, height=3.2cm]{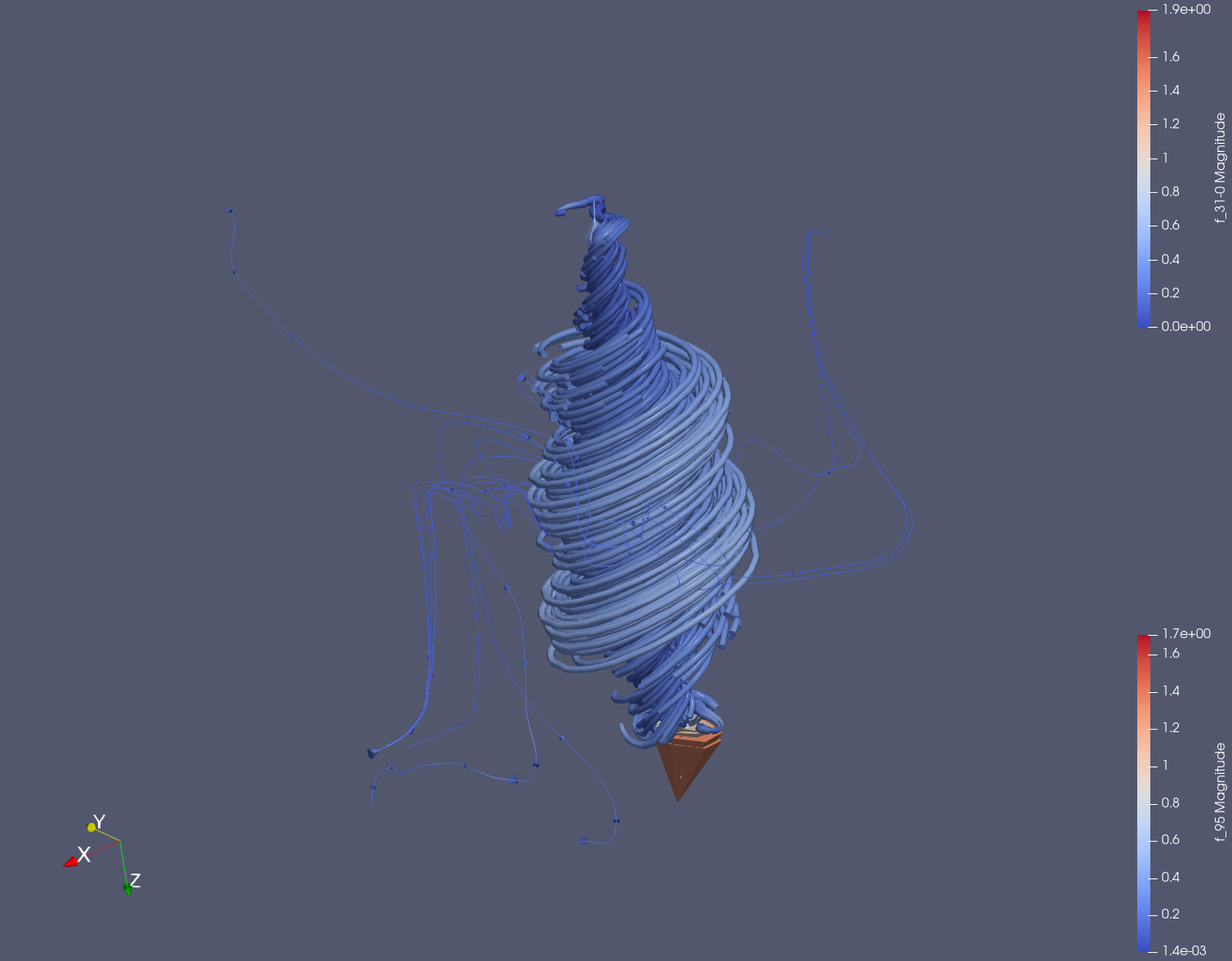}
\includegraphics[width=3.2cm, height=3.2cm]{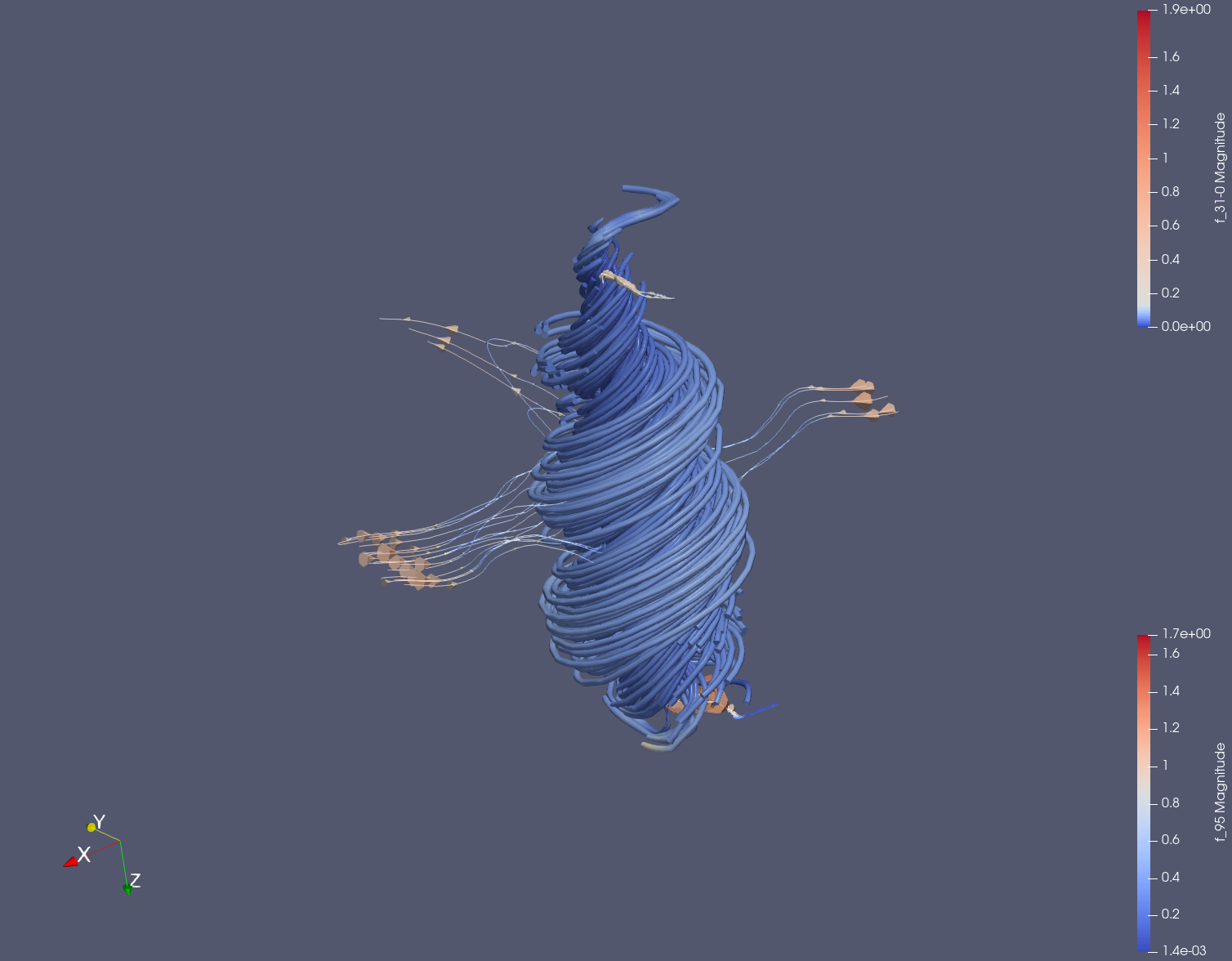} 
\includegraphics[width=3.2cm, height=3.2cm]{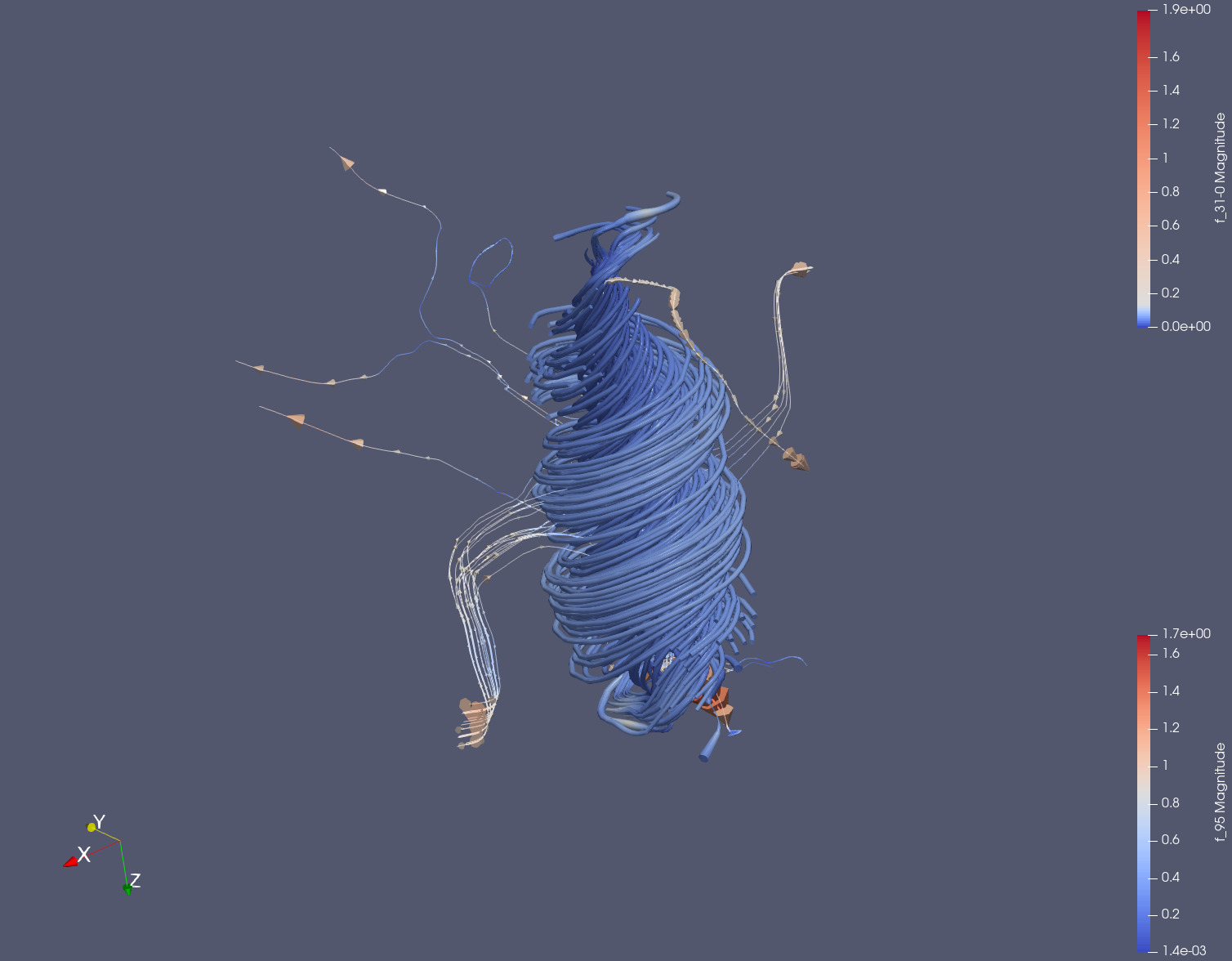}
\includegraphics[width=3.2cm, height=3.2cm]{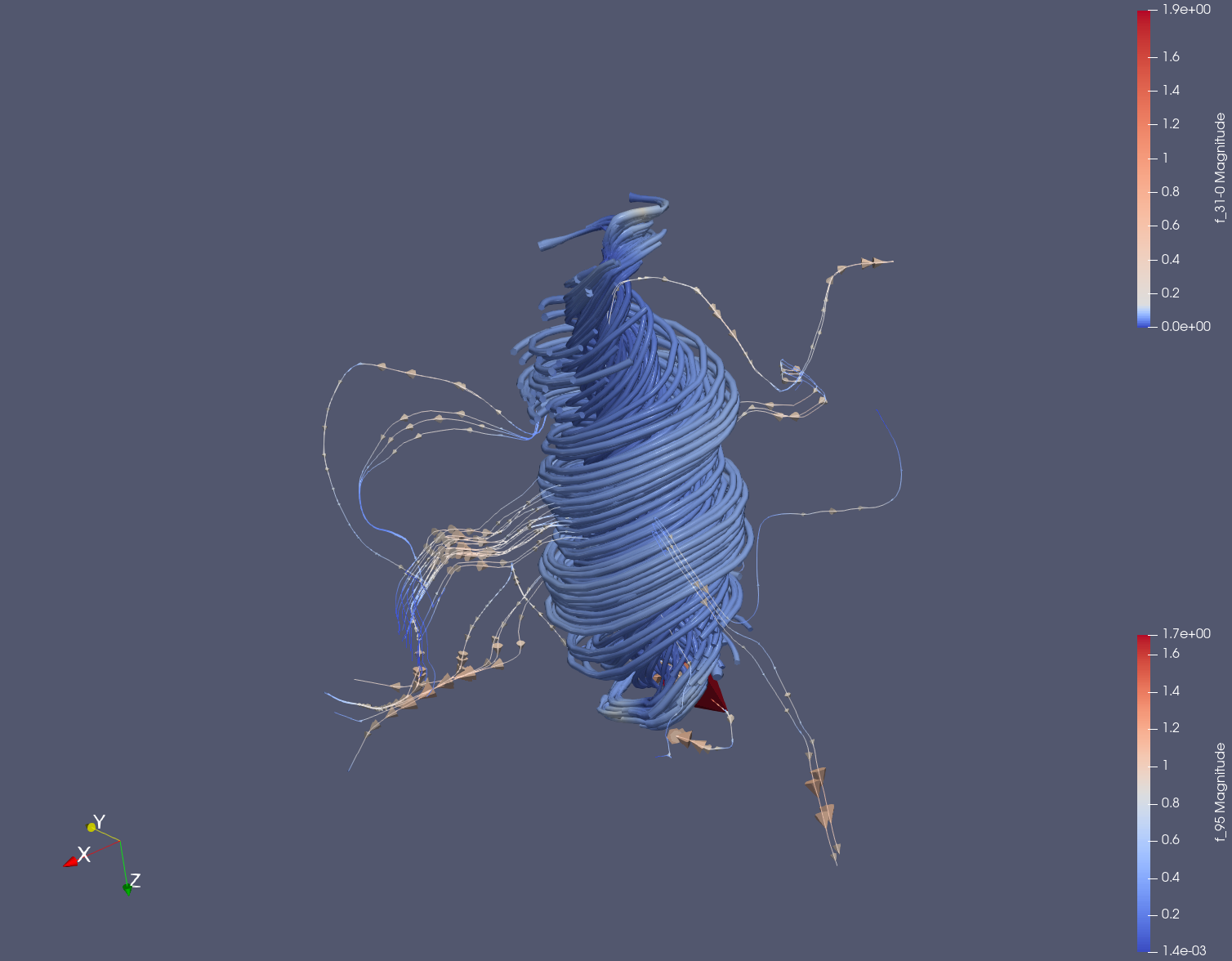} 
\includegraphics[width=3.2cm, height=3.2cm]{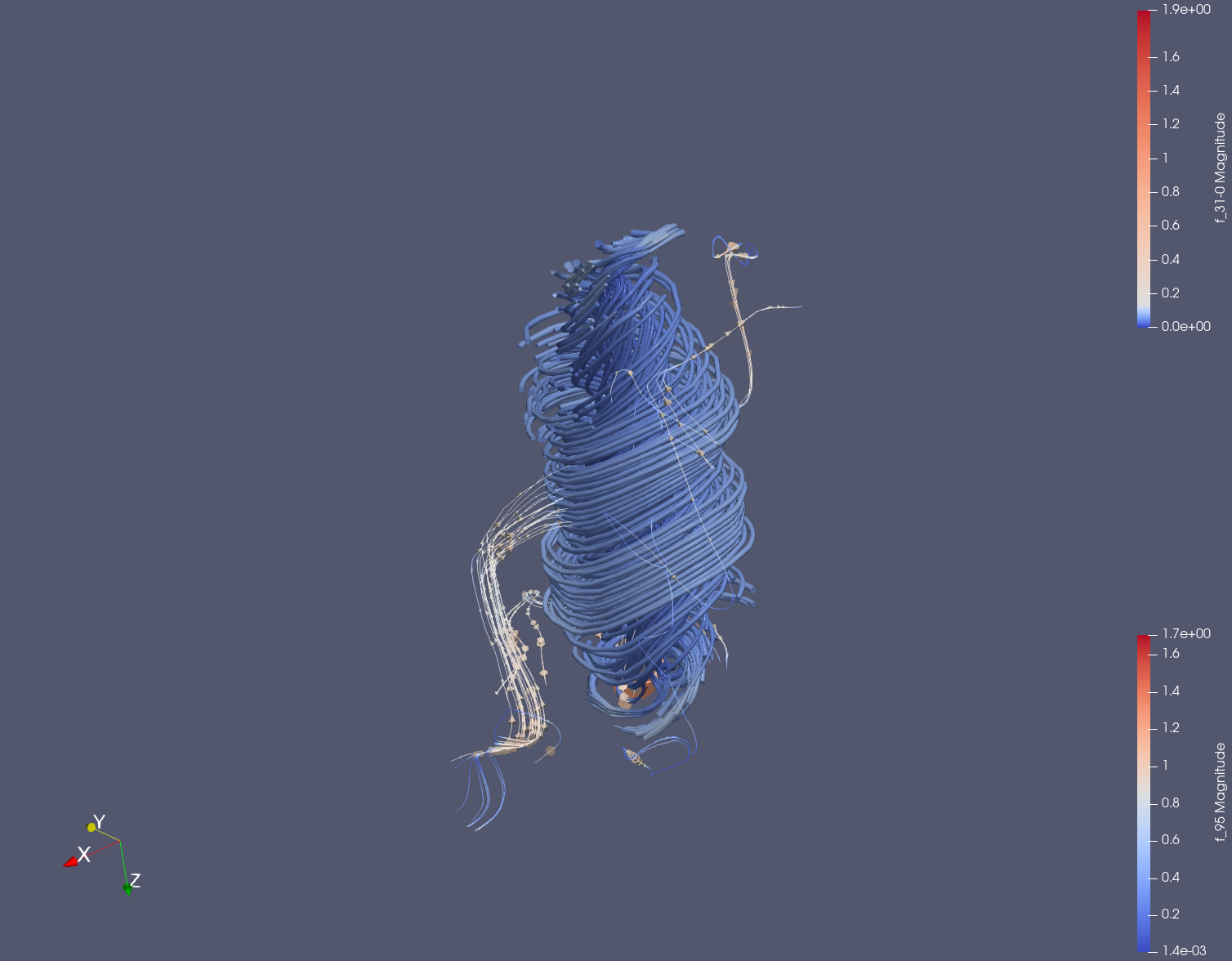}
\includegraphics[width=3.2cm, height=3.2cm]{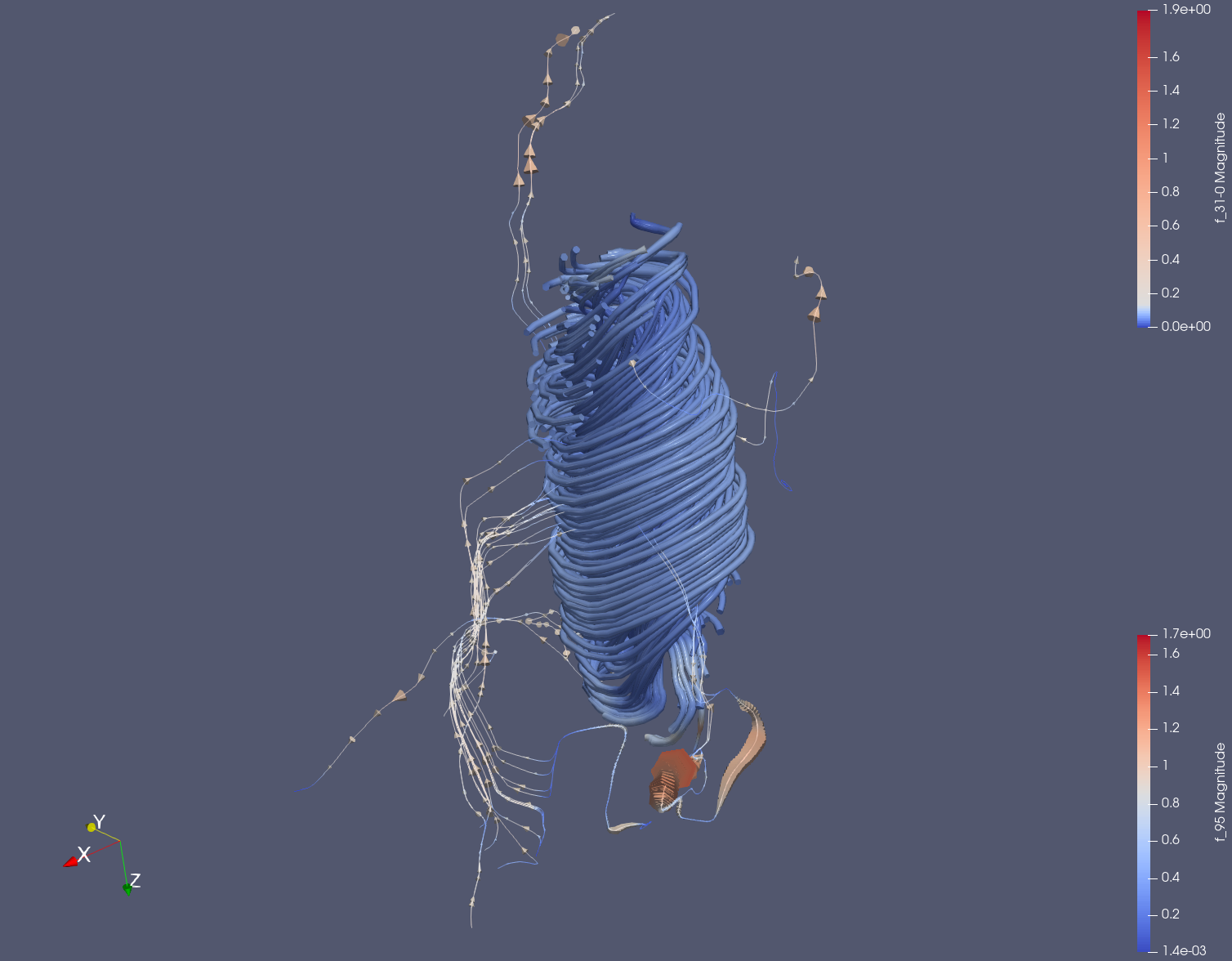} 
\includegraphics[width=3.2cm, height=3.2cm]{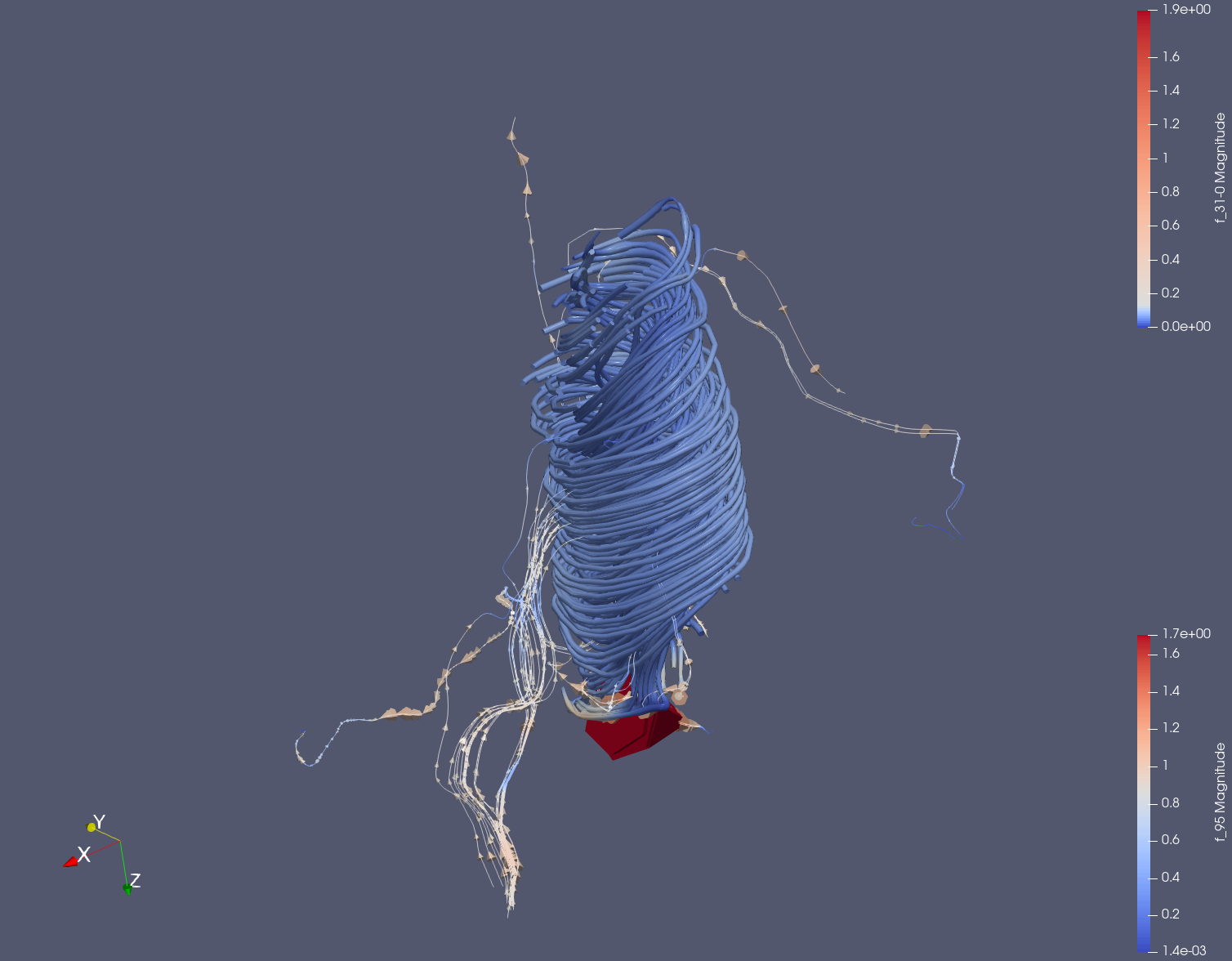}
\includegraphics[width=3.2cm, height=3.2cm]{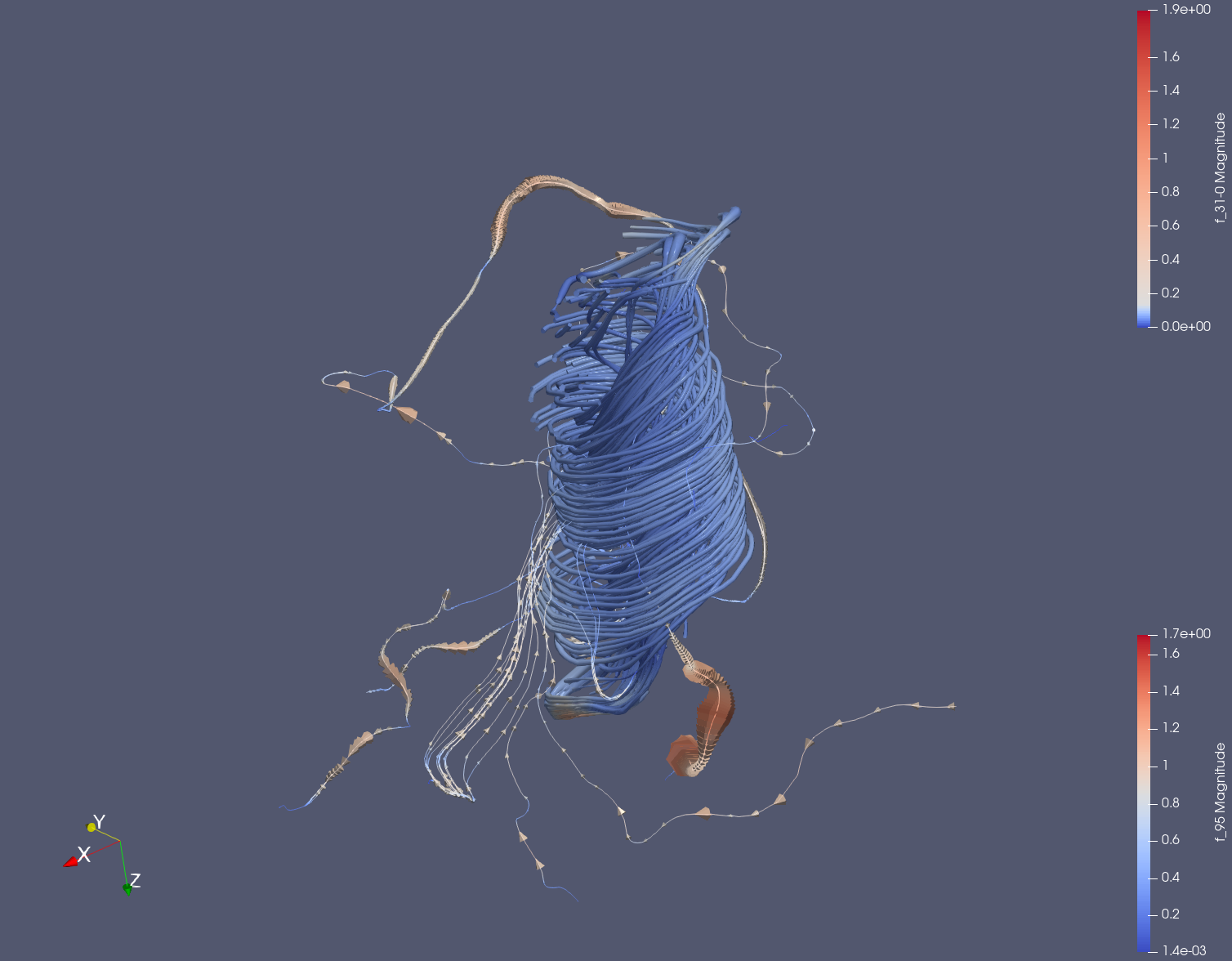} 
\includegraphics[width=3.2cm, height=3.2cm]{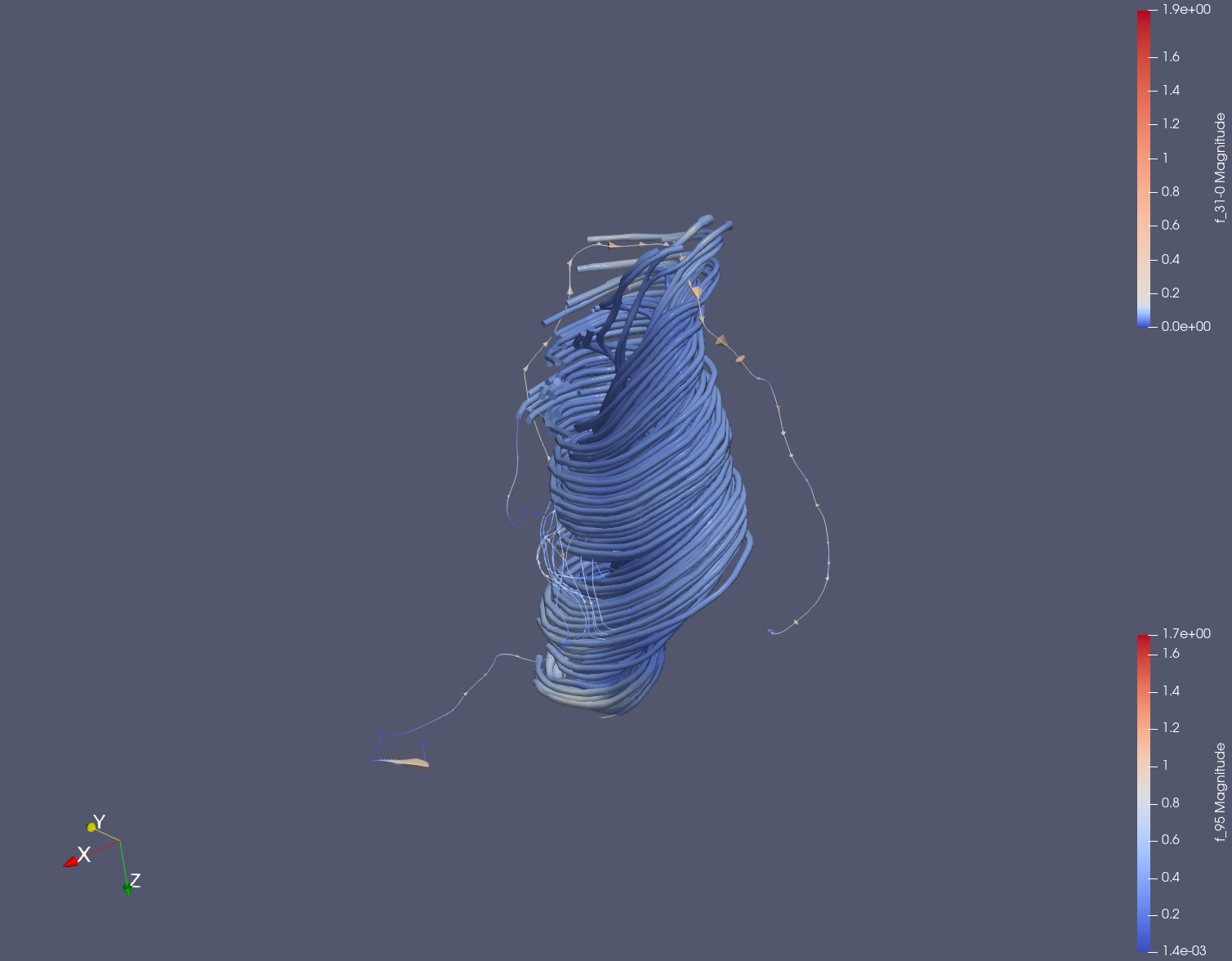}
\includegraphics[width=3.2cm, height=3.2cm]{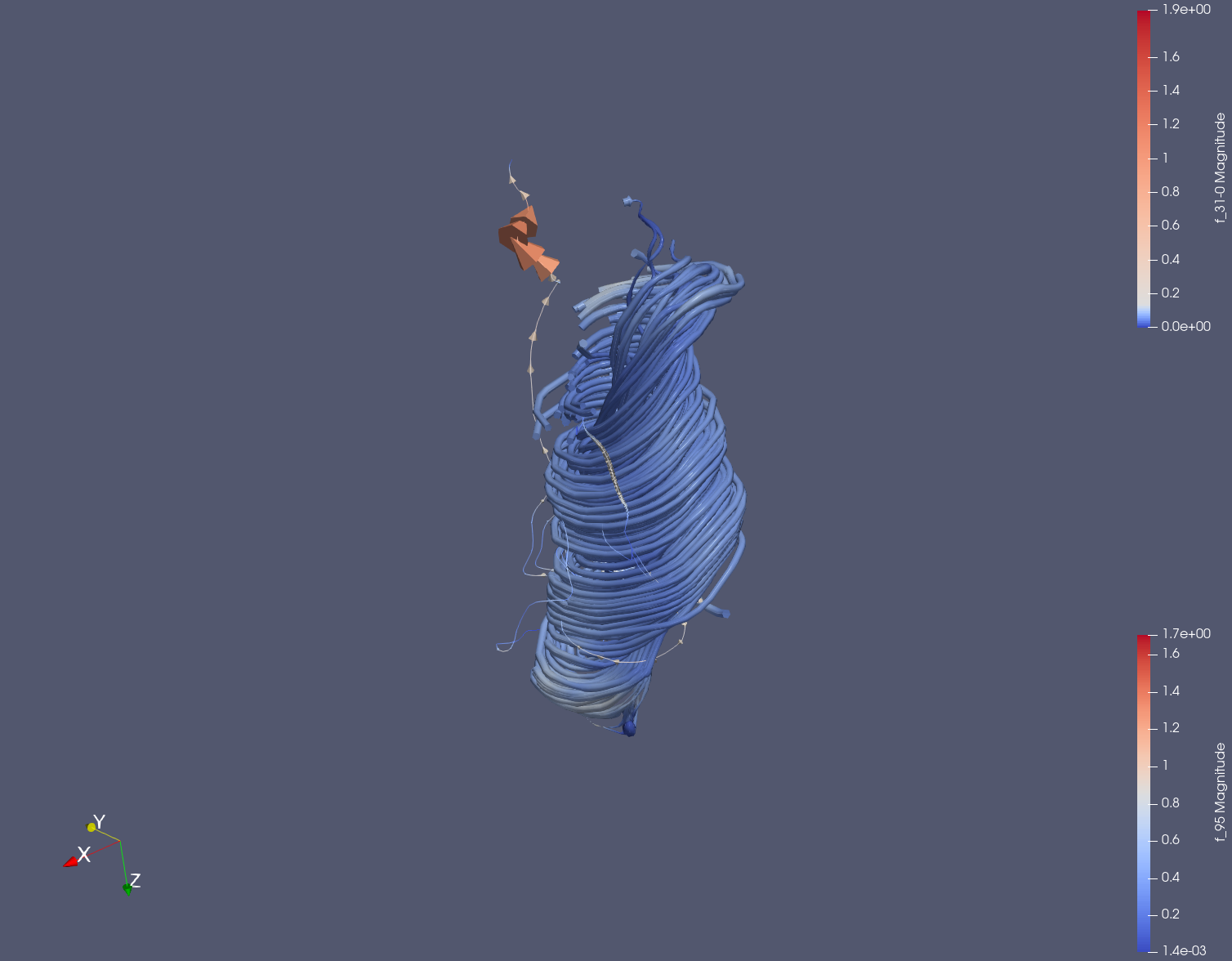}
\caption{Plot of stream tube of the magnetic field $\bm{B}_{h}$ obtained with Algorithm \ref{alg:main} with $R_e = \infty, R_m = \infty$, $h = 1/16$ and $\Delta t = 1/1000$. The lines with arrows show $\bm{u}_h$. The plot shows a sequence of magnetic tubes as time goes from top left to going right and then at the next row from left to right, each of which corresponds to time level, $T = 0.05, 0.1, 0.2, 0.3, 0.4, 0.5, 0.6, 0.7, 0.8, 0.9, 1.0$ and $1.1$.}
\label{fig:vel}
\end{figure}

\section{Conclusion}\label{sec:con}

We constructed finite element methods that preserve the discrete energy law, the magnetic Gauss law and the magnetic, cross helicity precisely at the discrete level. The construction relies on discrete de Rham complexes and mathematical properties of the MHD system.  In particular, the Lorentz force term and the magnetic advection term cancel with each other in the proof of the energy law, and the fluid and magnetic advection terms cancel with each other in the proof of the cross helicity conservation.  These cancelations reflect symmetry in the operator structures of the MHD system \cite{ma2016robust}, carried over from the continuous level to the discrete level, and can be important for the construction of efficient solvers \cite{ma2016robust}. 

For the spatial discretization, we used finite element differential forms, e.g., the N\'{e}d\'elec and Raviart-Thomas elements, in the numerical tests. The discussions in this paper also hold with other discrete de Rham complex, e.g., spline spaces \cite{buffa2011isogeometric}. 

To preserve the helicity and energy in the full discretization, we used the Crank-Nicolson scheme as the temporal discretization, c.f., \cite{rebholz2007energy}.  One can choose other temporal schemes that preserve quadratic invariants, c.f., \cite{hairer2006geometric}.

\section*{Acknowledgement}

Young-Ju Lee is supported in part by American Chemical Society PRF\# 57552-ND9.

The authors wish to thank Yang He, Ralf Hiptmair and Cecilia Pagliantini for helpful discussions.

 \appendix

 \section{Existence and uniqueness of solutions}



In this appendix, we discuss the existence and uniqueness of solutions to the nonlinear scheme. For technical issues and for simplicity, we will actually work on a slightly modified system other than \eqref{main:eq3d}. We first introduce some notation.

Define
$$
\bm{Z}_{h}:=\{\bm{v}\in H_{0}^{h}(\curl, \Omega) : \nabla_{h}\cdot \bm{v}_{h}=0\}, \quad H^{h}_{0}(\div0, \Omega):=\{\bm{C}_{h}\in H_{0}^{h}(\div, \Omega) : \nabla\cdot\bm{C}_{h}=0\}.
$$
\begin{lemma}\label{lem:divh0}
For $\bm{B}_{h}\in H_{0}^{h}(\div0, \Omega)$, we have $\nabla_{h}\cdot (\mathbb{Q}_{h}^{\curl}\bm{B}_{h})=0$.
\end{lemma}
\begin{proof}
Denote $\bm{H}_{h}:=\mathbb{Q}_{h}^{\curl}\bm{B}_{h}$. 
\begin{align*}
(\nabla_{h}\cdot \bm{H}_{h}, \nabla_{h}\cdot \bm{H}_{h})&=-(\nabla \nabla_{h}\cdot \bm{H}_{h}, \bm{H}_{h})=-(\nabla \nabla_{h}\cdot \bm{H}_{h}, \bm{B}_{h})\\&=(\nabla_{h}\cdot \bm{H}_{h}, \nabla\cdot \bm{B}_{h})=0.
\end{align*}
\end{proof}
For functions in $\bm{Z}_{h}$ and $H^{h}_{0}(\div0, \Omega)$, we recall the following Gaffney type inequalities \cite{he2019generalized}:  there exist positive constants $C$ such that
\begin{align}\label{poincare-1}
\|\bm{v}_{h}\|_{L^{3+\delta}}\leq C\|\nabla\times \bm{v}_{h}\|, \quad \forall \bm{v}_{h}\in \bm{Z}_{h}, 
\end{align}
\begin{align}\label{poincare-2}
\|\bm{B}_{h}\|_{L^{3+\delta}}\leq C\|\nabla_{h}\times \bm{B}_{h}\|, \quad \forall \bm{B}_{h}\in H^{h}_{0}(\div0). 
\end{align}
Here $\delta\in (0, 3]$ is a positive number depending on the regularity of the domain $\Omega\subset \mathbb{R}^{3}$. If $\Omega$ is convex or has $C^{1, 1}$ boundary, we can choose $\delta=3$. In the sequel, we assume that $\Omega$ is such that we can choose $\delta =1$, i.e., $\bm{Z}_{h}\hookrightarrow L^{4}(\Omega)$. 

{In the analysis below, we slightly modify the diffusion term in equations \eqref{main:eq3d} and assume that all the variables at the $n$-th time step are zero to avoid dealing with the cross terms between the $n$-th and the $(n+1)$-th time steps. The same analysis works for, e.g., a backward Euler time discretization.}

Consider the following variational form: find $(\bm{u}_{h}, \bm{B}_{h})\in \bm{Z}_{h}\times H^{h}_{0}(\div0, \Omega)$, such that for any $(\bm{v}_{h}, \bm{C}_{h})\in \bm{Z}_{h}\times H^{h}_{0}(\div0, \Omega)$,
\begin{align}\nonumber
(\Delta t)^{-1}&(\bm{u}_{h}, \bm{v}_{h})-(\bm{u}_{h}\times \mathbb{Q}_{h}^{\curl}(\nabla\times \bm{u}_{h}), \bm{v}_{h})+R_{e}^{-1}(\nabla\times \bm{u}_{h}, \nabla\times \bm{v}_{h})\\\label{reduced-1}
&\quad\quad-\text{c}((\nabla_{h}\times \bm{B}_{h})\times  \mathbb{Q}_{h}^{\curl}\bm{B}_{h}, \bm{v}_{h})=(\bm{F},  \bm{v}_{h}), \\\nonumber
(\Delta t)^{-1}(\bm{B}_{h}, \bm{C}_{h})-&(\bm{u}_{h}\times \mathbb{Q}_{h}^{\curl}\bm{B}_{h}, \nabla_{h}\times \bm{C}_{h})+1/2R_{m}^{-1}(\nabla\times \mathbb{Q}_{h}^{\curl}\bm{B}_{h}, \nabla\times \mathbb{Q}_{h}^{\curl} \bm{C}_{h})\\ 
\label{reduced-2} &\quad\quad\quad+\frac{1}{2}R_{m}^{-1}(\nabla_{h}\times \bm{B}_{h}, \nabla_{h}\times \bm{C}_{h})=(\bm{G}, \bm{C}_{h}).
\end{align}
Here $\bm{F}$ and $\bm{G}$ denote some general source terms. 

Comparing with \eqref{main:eq3d},  we modified the magnetic diffusion term in \eqref{reduced-2} by changing $\frac{1}{2}R_{m}^{-1}(\nabla_{h}\times \bm{B}_{h}, \nabla_{h}\times \bm{C}_{h})$ to $\frac{1}{2}R_{m}^{-1}(\nabla\times \mathbb{Q}_{h}^{\curl}\bm{B}_{h}, \nabla\times \mathbb{Q}_{h}^{\curl} \bm{C}_{h})$.

We write \eqref{reduced-1}-\eqref{reduced-2} in the following standard form: find $(\bm{u}_{h}, \bm{B}_{h})\in \bm{Z}_{h}\times H^{h}_{0}(\div0, \Omega)$, such that for any $(\bm{v}_{h}, \bm{C}_{h})\in \bm{Z}_{h}\times H^{h}_{0}(\div0, \Omega)$, 
\begin{equation}\label{trilinear}
a((\bm{u}_{h}, \bm{B}_{h}), (\bm{u}_{h}, \bm{B}_{h}); (\bm{v}_{h}, \bm{C}_{h})) =((\bm{F}, \text{c}\bm{G}), (\bm{v}_{h}, \bm{C}_{h})),
\end{equation}
where the trilinear form $a(\cdot; \cdot, \cdot)$ is defined by 
{
\begin{align*}
a((\bm{w}_{h}, \bm{K}_{h}), (\bm{u}_{h}, &\bm{B}_{h}); (\bm{v}_{h}, \bm{C}_{h})) :=(\Delta t)^{-1}(\bm{u}_{h}, \bm{v}_{h})-(\bm{w}_{h}\times \mathbb{Q}_{h}^{\curl}(\nabla\times \bm{u}_{h}), \bm{v}_{h})\\
&+R_{e}^{-1}(\nabla\times \bm{u}_{h}, \nabla\times \bm{v}_{h})-\text{c}((\nabla_{h}\times \bm{B}_{h})\times  \mathbb{Q}_{h}^{\curl}\bm{K}_{h}, \bm{v}_{h})+(\Delta t)^{-1}\text{c}(\bm{B}_{h}, \bm{C}_{h})\\&-\text{c}(\bm{w}_{h}\times \mathbb{Q}_{h}^{\curl}\bm{B}_{h}, \nabla_{h}\times \bm{C}_{h})+\frac{1}{2}\text{c}R_{m}^{-1}(\nabla_{h}\times \bm{B}_{h}, \nabla_{h}\times \bm{C}_{h})\\&+\frac{1}{2}\text{c}R_{m}^{-1}( \nabla\times \mathbb{Q}_{h}^{\curl}\bm{B}_{h}, \nabla_{h}\times \mathbb{Q}_{h}^{\curl} \bm{C}_{h}).
\end{align*}}
Introduce the following norm:
\begin{eqnarray*}
\|(\bm{u}_{h}, \bm{B}_{h})\|_{V}^{2} &:=& (\Delta t)^{-1}\|\bm{u}_{h}\|^{2}+(\Delta t)^{-1}\text{c}\|\bm{B}_{h}\|^{2} +R_{e}^{-1}\|\nabla\times \bm{u}_{h}\|^{2}\\ 
&&\quad + \text{c}R_{m}^{-1}\| \nabla_{h}\times \bm{B}_{h}\|^{2}+\text{c}R_{m}^{-1}\| \nabla\times \mathbb{Q}_{h}^{\curl}\bm{B}_{h}\|^{2}.
\end{eqnarray*} 
From Lemma \ref{lem:divh0}, \eqref{poincare-1} and \eqref{poincare-2}, we can further bound the following terms by the $\|\cdot\|_{V}$ norm:
$$
\|\bm{u}_{h}\|_{L^{4}}+\|\bm{B}_{h}\|_{L^{4}}+\|\mathbb{Q}_{h}^{\curl}\bm{B}_{h}\|_{L^{4}}\leq C\|(\bm{u}_{h}, \bm{B}_{h})\|_{V}.
$$

We include the existence theorem for nonlinear variational forms, which is given in, for example, \cite{Girault.V;Raviart.P.1986a}. Since we focus on the discrete level, we only state the results for finite dimensional problems.
\begin{theorem}\label{thm:nonlinear-monotone}
Assume that $\bm{V}$ is a finite dimensional vector space, and there exists a positive constant $\alpha$ such that a bounded trilinear form $a(\cdot;\cdot,\cdot)$ on $\bm{V}$ satisfies
$$
a(\bm{v}; \bm{v}, \bm{v})\geq \alpha \|\bm{v}\|^{2}, \quad \forall \bm{v}\in \bm{V}.
$$
Then the problem: given $\bm{f}\in \bm{V}^{\ast}$, find $\bm{u}\in \bm{V}$, such that for all $\bm{v}\in \bm{V}$, 
$$
a(\bm{u}; \bm{u}, \bm{v})=\bm{f}(\bm{v}),
$$ 
has at least one solution.
\end{theorem}

\begin{lemma}
The trilinear form $a(\cdot, \cdot, \cdot)$ is bounded, i.e., there exists a positive constant $C$ such that
$$
\left|a((\bm{u}_{h}, \bm{B}_{h}); (\bm{v}_{h}, \bm{C}_{h}), (\bm{w}_{h}, \bm{K}_{h}))\right |\leq C \|(\bm{u}_{h}, \bm{B}_{h})\|_{V}\|(\bm{v}_{h}, \bm{C}_{h})\|_{V}\|(\bm{w}_{h}, \bm{K}_{h})\|_{V}.
$$
\end{lemma}
\begin{proof}
It suffices to bound the nonlinear terms:
\begin{align*}
\left |(\bm{w}_{h}\times \mathbb{Q}_{h}^{\curl}(\nabla\times \bm{u}_{h}), \bm{v}_{h}) \right |&\leq \|\nabla\times \bm{u}_{h}\|\|\bm{w}_{h}\|_{L^{4}}\|\bm{v}_{h}\|_{L^{4}}\\
&\lesssim \|\nabla\times \bm{u}_{h}\|\|\nabla\times \bm{w}_{h}\|\|\nabla\times \bm{v}_{h}\|, 
\end{align*}
\begin{align*}
\left |((\nabla_{h}\times \bm{B}_{h})\times  \mathbb{Q}_{h}^{\curl}\bm{K}_{h}, \bm{v}_{h}) \right |&\leq \|\nabla_{h}\times \bm{B}_{h}\|\|\mathbb{Q}_{h}^{\curl}\bm{K}_{h}\|_{L^{4}}\|\bm{v}_{h}\|_{L^{4}}\\
&\lesssim \|\nabla_{h}\times \bm{B}_{h}\|\|\nabla\times \mathbb{Q}_{h}^{\curl}\bm{K}_{h}\|\|\nabla\times \bm{v}_{h}\|,
\end{align*}
and the estimate for $(\bm{w}_{h}\times \mathbb{Q}_{h}^{\curl}\bm{B}_{h}, \nabla_{h}\times \bm{C}_{h})$ is the same.

\end{proof}
\begin{lemma}
The trilinear form $a(\cdot, \cdot, \cdot)$ is coercive, i.e., there exists a positive constant $\alpha$ such that
$$
a((\bm{u}_{h}, \bm{B}_{h}); (\bm{u}_{h}, \bm{B}_{h}), (\bm{u}_{h}, \bm{B}_{h}))\geq \alpha \|(\bm{u}_{h}, \bm{B}_{h})\|_{V}^{2}.
$$
\end{lemma}
\begin{proof}
\begin{align*}
a((\bm{u}_{h}, \bm{B}_{h}); (\bm{u}_{h}, \bm{B}_{h}), (\bm{u}_{h},& \bm{B}_{h}))=(\Delta t)^{-1}\|\bm{u}_{h}\|^{2}+(\Delta t)^{-1}\text{c}\|\bm{B}_{h}\|^{2}+R_{e}^{-1}\|\nabla\times \bm{u}_{h}\|^{2}\\&
+\frac{1}{2}\text{c}R_{m}^{-1}\| \nabla_{h}\times \bm{B}_{h}\|^{2}+\frac{1}{2}\text{c}R_{m}^{-1}\| \nabla\times \mathbb{Q}_{h}^{\curl}\bm{B}_{h}\|^{2}.
\end{align*}
\end{proof}

We are now in a position to state the existence of the discrete variational form.
\begin{theorem}
For any $(\bm{F}, \bm{G})\in (\bm{Z}_{h})^{\ast}\times (H_{0}^{h}(\div0), \Omega)^{\ast}$, there exists at least one solution for \eqref{trilinear}. 
\end{theorem}
The uniqueness of solutions of \eqref{trilinear} with small data follows from standard argument, c.f., \cite{Girault.V;Raviart.P.1986a}. 

\bibliographystyle{siam}      
\bibliography{helicity}{}   

\end{document}